\documentclass[1 [leqno,11pt]{amsart}
\usepackage{amssymb, amsmath,amsmath,latexsym,amssymb,amsfonts,amsbsy, amsthm}
\numberwithin{equation}{section}

\allowdisplaybreaks

\numberwithin{equation}{section}

\def\sqr#1#2{{\vbox{\hrule height.#2pt
     \hbox{\vrule width.#2pt height#1pt \kern#1pt
           \vrule width.#2pt}
     \hrule height.#2pt}}}

\newtheorem{theorem}{Theorem}[section]

\newtheorem{proposition}[theorem]{Proposition}

\numberwithin{equation}{section}

\newcommand{\R}{\mathbb{R}}
\newcommand{\Us}{\mathbf{U}_s}
\newcommand{\U}{\mathbf{U}}
\newcommand{\W}{\mathbf{W}}
\newcommand{\UE}{U^\varepsilon}
\newcommand{\VE}{V^\varepsilon}

\newcommand{\e}{\varepsilon}
\newcommand{\T}{\mathcal{T}}
\newcommand{\X}{\mathbb{X}}
\newcommand{\Y}{\mathbb{Y}}
\newcommand{\Z}{\mathbb{Z}}

\newcommand{\z}{\langle}
\newcommand{\y}{\rangle}
\newcommand{\p}{\partial}
\newcommand{\se}{\sqrt\e}
\newcommand{\RE}{\mathbf{R}}
\newcommand{\F}{\mathbf{F}}
\newcommand{\ub}{\bar{u}}
\newcommand{\vb}{\bar{v}}

\newtheorem{Lemma}{Lemma}[section]

\newtheorem{Proposition}{Proposition}[section]
\newtheorem{Remark}{Remark}[section]

\begin{document}

\title[On the Steady Prandtl Boundary Layer Expansions]
{On the Steady Prandtl Boundary Layer Expansions}

\author{Chen Gao}
\address{Institute of Mathematics, AMSS, UCAS, Beijing 100190, China
123}
\email{gaochen@amss.ac.cn}

\author{Liqun Zhang}
\address{Hua Loo-Keng Key Laboratory of Mathematics, Institute of Mathematics, AMSS, and School of
Mathematical Sciences, UCAS, Beijing 100190, China 
123}
\email{lqzhang@math.ac.cn}

\begin{abstract}
In this paper, we consider the zero-viscosity limit of the 2D steady Navier-Stokes equations in $(0,L)\times\R^+$ with non-slip boundary conditions. By estimating the stream-function of the remainder, we justify the validity of the Prandtl boundary layer expansions.
\end{abstract}
\date{}
\maketitle

\section{Introduction}\label{sec:intro}
We consider the vanishing viscosity limit of steady Navier-Stokes equations
\begin{equation}\label{NSE}
\left\{
\begin{aligned}
&U^\e U^\e_X+V^\e U^\e_Y-\e\Delta U^\e+P^\e_X=0,\\
&U^\e V^\e_X+V^\e V^\e_Y-\e\Delta V^\e+P^\e_Y=0,\\
&U^\e_X+V^\e_Y=0,\\
&U^\e|_{Y=0}=V^\e|_{Y=0}=0,
\end{aligned}
\right.
\end{equation}
in a two dimensional domain $\Omega=\{(X,Y):0\leqslant X\leqslant L,Y\geqslant0\}.$ A formal limit $\e\rightarrow 0$ should lead to the Euler flow $[U^0,V^0]$ inside $\Omega$:
\begin{equation}\label{Euler}
\left\{
\begin{aligned}
&U^0 U^0_X+V^0 U^0_Y+P^0_X=0,\\
&U^0 V^0_X+V^0 V^0_Y+P^0_Y=0,\\
&U^0_X+V^0_Y=0,\\
&V^0|_{Y=0}=0.
\end{aligned}
\right.
\end{equation}
Generically, there is a mismatch between the tangential velocities of the Euler flow $U_{0}(X,0)\neq0$ and the prescribed Navier-Stokes flows $\UE(X,0) = 0$ on the boundary, because of the difference of boundary conditions imposed on the two systems.

Due to the mismatch on the boundary, Prandtl in 1904, proposed a thin fluid
boundary layer of size $\se$ to connect different velocities
$U_{0}(X,0)$ and $0.$  In the following discussion, we shall make use of the scaled boundary layer, or
Prandtl's variables:
\begin{align}\label{var change}
x = X, \qquad y=\frac{Y}{\se}.
\end{align}

In these variables, we express the solution of the NS equation $[\UE,\VE]$ via $[u^\e,v^\e]$ as
\[
\begin{split}
[\UE(X,Y),\VE(X,Y)]=[u^\e(x,y),\se v^\e(x,y)]
\end{split}
\]
in which we note that the scaled normal velocity $v^\e$ is $\frac{1}{\se}$ of the original velocity $\VE$. Similarly, $P^\e(X,Y) = p^\e(x,y)$. In these new variables, the Navier-Stokes equations (\ref{NSE}) now read
\begin{equation}\label{P-NSE}
\left\{
\begin{aligned}
u^\e u^\e_x+v^\e u^\e_y+p^\e_x&=u^\e_{yy}+\e u^\e_{xx},\\
\e[u^\e v^\e_x+v^\e v^\e_y]+p^\e_y&=\e[v^\e_{yy}+\e v^\e_{xx}],\\
u^\e_x+v^\e_y&=0.
\end{aligned}
\right.
\end{equation}
Let $\e\rightarrow0$, it leads to the Prandtl equations:
\begin{equation}\label{Prandtl}
\left\{
\begin{aligned}
&u^0_pu^0_{px}+v^0_pu^0_{py}-u^0_{pyy}+p^0_{px}=0,\\
&p^0_{py}=0,\\
&u^0_{px}+v^0_{py}=0.\\
\end{aligned}
\right.
\end{equation}
Prandtl hypothesized that when viscosity $\e$ is  the Navier-Stokes flow can be approximately decomposed into two parts:
\begin{align}\label{p-exp}
\begin{aligned}
&\UE(X,Y)\approx u^0_e(X,Y)-u^0_e(X,0)+u^0_p(X,\frac{Y}{\se}),\\
&\VE(X,Y)\approx v^0_e(X,Y)+\se v^0_p(X,\frac{Y}{\se}),
\end{aligned}
\end{align}
in which $(u^0_e,v^0_e)$ denotes the Euler flow.

We attempt to verify the Prandtl boundary layer expansion (\ref{p-exp}) under more general conditions. Now let us review the main problems in boundary layer theory. Two important open problems in this area are the well-posedness of Prandtl equations and the justification of viscosity vanishing limits. The first problem is relatively well understood and proved the well-posedness in some cases. Sammartino and Caflisch \cite{Sam1} obtained their result for the analytic class. For the monotonic data, Oleinik and Samokhin \cite{Oleinik} obtained the local existence of classical solutions of 2D Prandtl equations by using the Crocco transformation. Xin and Zhang \cite{Zhang} proved the global existence of weak solutions to this system for the favorable pressure, which by their regularity result are classical solutions. Lately, Alexandre \cite{Alex} and Masmoudi and Wong \cite{Mas} independently proved by the energy methods, the local well-posedness of prandtl equations in Sobolev space under the monotonic assumptions. Meanswhile, Liu, Wang and Yang generalized some results in 3D case with special structure. There are some results of well-posedness in Gervey class \cite{DG} and \cite{GM}. On the hands, Gerard-Varet and Dormy \cite{Gerard-ill} established the linearized ill-posedness without the monotonicity condition in Sobolev space for Prandtl equations. There are some relevant results in \cite{E97} and \cite{GN12}.

For the second problem, the verification of the viscosity vanishing limits is more difficult and remains a challenge problem in general. The problem in the analytic case was proved in \cite{Sam2} and \cite{WWZ17}. The problem in 2D case was studied by Maekawa \cite{Mae} and proved the convergence under the assumption on the initial vorticity vanishing in the neighbourhood of boundary. The auther in \cite{GMM18} established the Gervey stability for shear flows. There are some results of instablity in Sobolev space, cf. \cite{GGN16}-\cite{GN18}. For the steady case, an important progress was made by Guo, Nguyen \cite{GN} and Iyer \cite{I17} for Prandtl boundary layer expansions for the steady Navier-Stokes flows over a moving boundary. Especially Guo and Iyer \cite{GI} proved very recently, the convergence result for no-slip boundary conditions in shear Euler flows in the case the length of the region $L$ is small. Meanwhile, Gerard-Varet and Maekawa \cite{Gerard pra} obtained stability of shear flows of Prandtl type in some class in Sobolev space . Those results are the great inspirations to us.

In our first result, we assume that the outside Euler flow $\lbrack U^{0},V^{0}]\equiv (u^0_e(X,Y),v^0_e(X,Y))$ satisfying the following hypothesis:
\begin{align}\label{Eul profile}
&0<c_0\leqslant u^0_e\leqslant C_0<\infty,\\\label{small}
&\|v^0_e\|_{L^\infty}\ll1,\\\label{Eul v}
&\|\z Y\y^k\nabla^m v^0_e\|_{L^\infty}<\infty \text{ for sufficiently large } k, m\geqslant0,\\\label{Eul u}
&\|\z Y\y^k\nabla^m u^0_e\|_{L^\infty}<\infty \text{ for sufficiently large } k, m\geqslant1.
\end{align}
Here $\z Y\y=Y+1$.

We consider the Prandtl equations with the positive data.
\begin{equation}\label{pra them}
\left\{
\begin{aligned}
&u^0_pu^0_{px}+v_p^0u^0_{py}-u^0_{pyy}+p^0_{px}=0,\hspace{3 mm}p^0_{py}=0,\hspace{3 mm}u_{px}^0+v_{py}^0=0,\hspace{3 mm}(x,y)\in(0,L)\times\mathbb{R}_+,\\
&u^0_p|_{x = 0} = U^0_P(y), \hspace{5 mm} u^0_p|_{y = 0} = v^0_p|_{y = 0} = 0, \hspace{5 mm} u^0_p|_{y \uparrow \infty} = u^0_e|_{Y = 0}.
\end{aligned}
\right.
\end{equation}
$U^0_P$ is a prescribing smooth function such that
\begin{align}
\begin{aligned} \label{pos}
& U^0_P > 0 \text{ for } y > 0, \hspace{3 mm} \partial_y U^0_P(0) > 0, \hspace{3 mm} \partial_y^2 U^0_P-u^0_eu^0_{ex}(0) \sim y^2 \text{ near } y = 0,\\
&\partial_y^m \{U^0_P - u^0_e(0)\} \text { decay fast for any }m\geqslant0.
\end{aligned}
\end{align}
In fact, under above conditions on $U^0_P$, equations(\ref{pra them}) admit a classical solution $[u^0_p,v^0_p]$. Now we state our first result:
\begin{theorem}\label{main1}  Assume the Euler flows $[u^0_e,v^0_e]$ satisfy (\ref{Eul profile})-(\ref{Eul u}), $U^0_P$ is a smooth function satisfying (\ref{pos}) and high order compatibility conditions, $L$ is a constant small enough,
\noindent then there exist $C(L), \e_0(L)>0$ depending on $L$, such that for $0<\e\leqslant\e_0$, equations (\ref{NSE}) admits a solution $[U^\e,V^\e]\in W^{2,2}(\Omega)$, satisfying:
\begin{align}
\begin{aligned}
&\|U^\e-u^0_e+u^0_e|_{Y=0}-u^0_p\|_{L^\infty}\leqslant C\se,\\
&\hspace{1cm}\|V^\e-v^0_e\|_{L^\infty}\leqslant C\se,
\end{aligned}
\end{align}
with the following boundary conditions:
\begin{equation}\label{Boundary C1}
\begin{aligned}
&[U^\e, V^\e]|_{Y=0}=0,\\
&[U^\e, V^\e]|_{X=0}=[u^0_e(0,Y)-u^0_e(0,0)+u^0_p(0,\frac{Y}{\se})+\se a_0, v^0_e(0,Y)+\se b_0],\\
&[U^\e, V^\e]_{X=L}=[u^0_e(L,Y)-u^0_e(L,0)+u^0_p(L,\frac{Y}{\se})+\se a_L, v^0_e(L,Y)+\se b_L].
\end{aligned}
\end{equation}
Here
\begin{equation}
\begin{aligned}
&a_0(Y)=u^1_e(0,Y)+u^1_b(0,\frac{Y}{\se})+\se u^2_e(0,Y)+\se \hat{u}^2_b(0,\frac{Y}{\se}),\\
&a_L(Y)=u^1_e(L,Y)+u^1_b(L,\frac{Y}{\se})+\se u^2_e(L,Y)+\se \hat{u}^2_b(L,\frac{Y}{\se}),\\
&b_0(Y)=v^0_b(0,\frac{Y}{\se})+v^1_e(0,Y)+\se v^1_b(0,\frac{Y}{\se})+\se v^2_e(0,Y)+\e \hat{v}^2_b(0,\frac{Y}{\se}),\\
&b_L(Y)= v^0_b(L,\frac{Y}{\se})+v^1_e(L,Y)+\se v^1_b(L,\frac{Y}{\se})+\se v^2_e(L,Y)+\e \hat{v}^2_b(L,\frac{Y}{\se}),\\
\end{aligned}
\end{equation}
are smooth functions constructed in Proposition \ref{construct}.
\end{theorem}

 For the second result, we consider $L$ is any given positive constant. We assume the Euler flow $\lbrack U^{0},V^{0}]\equiv \lbrack u^0_e(Y),0]$ is a shear flow, that is, it satisfies the following hypothesis:
\begin{align}
\begin{aligned}\label{shear Eul profile}
&0<c_0\leqslant u^0_e\leqslant C_0<\infty,\\
&\|\z Y\y^k\nabla^m u^0_e\|_{L^\infty}<\infty \text{ for sufficiently large } k, m\geqslant1.
\end{aligned}
\end{align}
Here $\z Y\y=Y+1$.

While we assume $[u^0_p,v^0_p]$ is a smooth solution of Prandtl equations (\ref{pra them}) satisfying the following hypothesis:
\begin{equation}\label{pra them2}
\begin{aligned}
&u^0_p > 0,\hspace{1mm}-u^0_{pyy}\geqslant0, \text{ for } y > 0,\\
&u^0_{py}>0, \text{ for } y \geqslant 0,\\
&\nabla^m \{u^0_p - u^0_e(0)\} \text { decay fast for any }m\geqslant0.
\end{aligned}
\end{equation}
Because the Euler flow here is independent on $x$, by the classical result of Oleinik \cite{Oleinik}, for any given $L>0$, this kind of solutions exist.
The important example is the famous Blasius's self-similar solution. 

Now we state our second result:
\begin{theorem}\label{main2}  Assume the Euler flows $[u^0_e,v^0_e]$ satisfy (\ref{shear Eul profile}), the Prandtl profiles $[u^0_p,v^0_p]$ satisfy (\ref{pra them2}), and $L>0$ is any given constant,
\noindent then there exist $C(L), \e_0(L)>0$ depending on $L$, such that for $0<\e\leqslant\e_0$, equations (\ref{NSE}) admits a solution $[U^\e,V^\e]\in W^{2,2}(\Omega)$, satisfying:
\begin{align}
\begin{aligned}
&\|U^\e-u^0_e+u^0_e|_{Y=0}-u^0_p\|_{L^\infty}\leqslant C\se,\\
&\hspace{1cm}\|V^\e\|_{L^\infty}\leqslant C\se,
\end{aligned}
\end{align}
with the boundary conditions:
\begin{equation}\label{Boundary C2}
\begin{aligned}
&[U^\e, V^\e]|_{Y=0}=0,\\
&[U^\e, V^\e]|_{X=0}=[u^0_e(Y)-u^0_e(0)+u^0_p(0,\frac{Y}{\se})+\se a_0, \se b_0],\\
&[U^\e, V^\e]_{X=L}=[u^0_e(Y)-u^0_e(0)+u^0_p(L,\frac{Y}{\se})+\se a_L, \se b_L].
\end{aligned}
\end{equation}
Here
\begin{equation}
\begin{aligned}
&a_0(Y)=u^1_e(0,Y)+u^1_b(0,\frac{Y}{\se})+\se u^2_e(0,Y)+\se \hat{u}^2_b(0,\frac{Y}{\se}),\\
&a_L(Y)=u^1_e(L,Y)+u^1_b(L,\frac{Y}{\se})+\se u^2_e(L,Y)+\se \hat{u}^2_b(L,\frac{Y}{\se}),\\
&b_0(Y)=v^0_b(0,\frac{Y}{\se})+v^1_e(0,Y)+\se v^1_b(0,\frac{Y}{\se})+\se v^2_e(0,Y)+\e \hat{v}^2_b(0,\frac{Y}{\se}),\\
&b_L(Y)= v^0_b(L,\frac{Y}{\se})+v^1_e(L,Y)+\se v^1_b(L,\frac{Y}{\se})+\se v^2_e(L,Y)+\e \hat{v}^2_b(L,\frac{Y}{\se}),\\
\end{aligned}
\end{equation}
are smooth functions constructed in Proposition \ref{construct}.
\end{theorem}

The theorem shows if the expansions (\ref{p-exp}) are right on $\partial\Omega$, we can justify they are right in $\Omega$. The first result is different from \cite{GI}, because we actually  prescribe Dirichlet boundary conditions on the solution. While Guo and Iyer \cite{GI} posed the solution of Navier-Stokes equations with some Nuemann conditions. Moreover, we can prove the convergence in non-shear Euler flow case. The second one is not a local version, $L$ can be chosen for any large constant, which needs to overcome additional difficulties. In this situation, we can also deal with non-shear Euler case, when the Euler flow satisfies (\ref{Eul profile}), (\ref{Eul v}), (\ref{Eul u}), and
\begin{align}\label{small2}
&\|v^0_e\|_{L^\infty}+\|\nabla v^0_e\|_{L^\infty}\leqslant \delta(L),
\end{align}
where $\delta(L)$ is a small constant depending on $L$. The proof is similar to theorem \ref{main2}.

To prove the the theorems, we first construct the approximate solutions $\U_s=[U_s,V_s]$ of Navier-Stokes equations. The main difficulty is estimating the remainders $\U:=\U^\e-\U_s$. While $\U$ satisfies the following linearized Navier-Stokes equations:
\begin{align}
-\e\Delta \U+\Us\cdot\nabla\U+\U\cdot\nabla\Us+\nabla P=\F.
\end{align}
The method of \cite{GI} is by taking the partial derivatives of the vorticity equations respect to $x$, they find the Rayleigh term and bi-Laplacian terms enjoy good interaction properties. What we choose to estimate is the stream-function of $\U$ which has natural boundary conditions. We can also estimate the second derivatives of stream-function which can be dominating by $\F$, which essentially leads to the proof the theorem. The second results need more subtle calculations, under some monotonicity assumptions on the solution of Prandtl's equations, we obtain the similar estimates.

This paper is organized as follows: In Section 2, we construct the approximation solutions by the asymptotic expansion method. In Section 3, we estimate the stream-function of remainder. In Section 4, we prove the main theorems.

\section{Construction of the approximate solution}
To construct the approximate solutions, we follow the idea in \cite{GI}. We will need higher order approximations, as compared to (\ref{p-exp}), in order to control the remainders. Precisely, we search for approximate solutions of the Navier-Stokes equations in the following form:
\begin{equation}\label{proflie}
\begin{aligned}
U^\e(X,Y)\approx&u^0_e(X,Y)+u^0_b(X,\frac{Y}{\se})+\se[u^1_e(X,Y)+u^1_b(X,\frac{Y}{\se})]\\
         &+\e[u^2_e(X,Y)+u^2_b(X,\frac{Y}{\se})],\\
V^\e(X,Y)\approx&v^0_e(X,Y)+\se[v^0_b(X,\frac{Y}{\se})+v^1_e(X,Y)]+\e[v^1_b(X,\frac{Y}{\se})+v^2_e(X,Y)]\\
         &+\e^\frac{3}{2}v^2_b(X,\frac{Y}{\se}),\\
P^\e(X,Y)\approx&p^0_e(X,Y)+p^0_b(X,\frac{Y}{\se})+\se[p^1_e(X,Y)+p^1_b(X,\frac{Y}{\se})]\\
         &+\e[p^2_e(X,Y)+p^2_b(X,\frac{Y}{\se})]+\e^\frac{3}{2}p^3_b(X,\frac{Y}{\se}),
\end{aligned}
\end{equation}
in which $[u^j_e,v^j_e,p^j_e]$ and $[u^j_b,v^j_b,p^j_b]$, with $j=0,1,2,$ denoting the Euler profiles and boundary layer profiles, respectively. Here, we note that these profile solutions also depend on $\e$. And the Euler flows are always evaluated at $(X,Y)$, whereas the boundary layer profiles are at $(X,\frac{Y}{\se})$.

For convenience, we will introduce some notation here. we write $$\z\cdot,\cdot\y=\z\cdot,\cdot\y_{L^2_{X,Y}}, $$ $$\z\cdot,\cdot\y_{Y=0}=\z\cdot,\cdot\y_{L^2_X(Y=0)},$$ $$\|\cdot\|=\|\cdot\|_{L^2_{X,Y}}$$ and $$\|\cdot\|_{\infty}=\|\cdot\|_{L^\infty_{X,Y}}=\|\cdot\|_{L^\infty_{x,y}}.$$ We denote $a\lesssim b$ means there exist a positive constant $C_0$, s.t. $a\leqslant C_0b$, here $C_0$ is independent on $\se$ and $L$ when $L\leqslant 1$, and $C_0$ is independent on $\se$ but dependent on $L$ when $L\geqslant1$. The notation $a\lesssim_L b$ means that there exist a positive constant $C_0(L)$, such that $a\leqslant C_0(L)b$, here $C_0(L)$ is independent on $\se$ but dependent on $L$. And we let $a=O\big(b\big)$ denote $|a|\lesssim b$.

\subsection{The zeroth-order profiles}
Recall the Euler flow $[u^0_e,v^0_e]$. Let $\psi$ be the stream-function of $[u^0_e,v^0_e]$
\begin{align}
\psi(X,Y):=\int_0^Yu^0_e(X,Y')dY', \hspace{3 mm} \psi_Y=u^0_e, \hspace{3 mm} \psi_X=-v^0_e,
\end{align}
then Euler equations (\ref{Euler}) are equivalent to:
\begin{align}
\Delta\psi=F_e(\psi).
\end{align}
From the assumptions in (\ref{Eul profile})-(\ref{Eul u}), we can know that $F_e$ together with sufficiently many derivatives are bounded and decaying in its argument.

While for Prandtl equations, there is a famous result due to Oleinik \cite{Oleinik}:
\begin{Proposition}[Oleinik]   \label{Oleinik} Assume boundary data  $U^0_P \in C^\infty$ satisfies (\ref{pos}), then for some $L > 0$, equations (\ref{pra them}) exists a solution $[u^0_p, v^0_p]$, satisfying, for some $y_0, m_0 > 0$,
\begin{equation}\label{OL}
\begin{aligned}
&\sup_{(0,L)\times(0,\infty)} |u^0_p, \p_y u^0_p, \p_{yy}u^0_p, \p_x u^0_p| \lesssim 1, \\
&\inf_{(0,L)\times (0, y_0)} \p_y u^0_p > m_0 > 0.
\end{aligned}
\end{equation}
\end{Proposition}
In fact, if $U^0_P$ satisfies high order parabolic compatibility conditions at the corner $(0,0)$, then $[u^0_p,v^0_p]$ are smooth enough. Next, let us consider the parabolic compatibility conditions of Prandtl's equations (\ref{pra them}). By the idea of \cite{GI-H},
\[
\begin{split}
&u^0_pu^0_{px}+v_p^0u^0_{py}-u^0_{pyy}+p^0_{px}=0,\hspace{3 mm}p^0_{py}=0,\hspace{3 mm}u_{px}^0+v_{py}^0=0,\hspace{3 mm}(x,y)\in(0,L)\times\mathbb{R}_+,\\
&u^0_p|_{x = 0} = U^0_P(y), \hspace{5 mm} u^0_p|_{y = 0} = v^0_p|_{y = 0} = 0, \hspace{5 mm} u^0_p|_{y \uparrow \infty} = u^0_e|_{Y = 0}.
\end{split}
\]
Due to $u^0_p|_{y = 0} = v^0_p|_{y = 0} = 0$ and $u_{px}^0+v_{py}^0=0$, $v(x,y)\sim y^2 \text{ near } y = 0$. By Bernoulli's law, $p^0_{px}(x,y)+u^0_e(x,0)u^0_{ex}(x,0)=0$, we can evaluate the Prandtl's equations on $y=0$, then
$$\partial_y^2 U^0_P(y)-u^0_e(0,0)u^0_{ex}(0,0) \sim y^2 \text{ near } y = 0.$$
Take partial derivative of the Prandtl's equations respect to $x$:
\begin{equation}\label{partial Pra}
\begin{aligned}
&u^0_pu^0_{pxx}+u^0_{px}u^0_{px}+v^0_{px}u^0_{py}+v_p^0u^0_{pxy}+v^0_{pyyy}+p^0_{pxx}=0.
\end{aligned}
\end{equation}
By evaluating (\ref{partial Pra}), we have
\begin{align}\label{compati}
\begin{aligned}
v^0_{pyyy}(0,y)-u^0_e(x,0)u^0_{exx}(0,0)-u^0_{ex}(0,0)u^0_{ex}(0,0) \sim y^2 \text{ near } y = 0.
\end{aligned}
\end{align}
We can solve the $v^0_p(0,y)$ from $U^0_P(Y)$ by Prandtl's equation:
\[
\begin{split}
&-u^0_pv^0_{py}+v^0_pu^0_{py}=u^0_{pyy}-p^0_{px},\\
\end{split}
\]
then
\begin{align}\label{pra v bry}
v^0_p(0,y)=-U^0_P(y)\int_0^y\frac{\p^2_yU^0_P-p^0_{px}(0,0)}{(U^0_P)^2}dy'.
\end{align}
The above integral is well-defined, because $\p_yU^0_P(0)>0$. So we obtain the compatibility conditions on $U^0_P$ from (\ref{compati}) and (\ref{pra v bry}):
\begin{align}\label{compati h}
\begin{aligned}
&\p^5_yU^0_P|_{y=0}-2u^0_{ex}(0,0)u^0_{ex}(0,0)\p_yU^0_P|_{y=0}=0,\\
&[\p_yU^0_P\p^6_yU^0_P-\p^2_yU^0_P\p^5_yU^0_P+2\p^3_yU^0_P\p^4_yU^0_P]|_{y=0}=0.
\end{aligned}
\end{align}
By continuing to take partial derivative with respect to $x$, we can obtain high order parabolic compatibility conditions on $U^0_P$.

Following the proof of Oleinik in \cite{Oleinik}, we have:
\begin{Lemma}\label{pra0}
If $U^0_P \in C^\infty$ satisfies (\ref{pos}) and high order parabolic compatibility conditions, then
\begin{align}
\|\z y\y^M\nabla^k(u^0_p(x,y)-u^0_e(x,0))\|_{\infty}\lesssim1, \hspace{2mm}\text{  for  }\hspace{2mm} 0\leqslant k \leqslant K,
\end{align}
here $K$ and $M$ are constants.
\end{Lemma}

By using the Crocco transformation in \cite{Oleinik}, we have 
\begin{Proposition}\label{pra02}
Assume $u^0_e|_{Y=0}$ is indenpendent on $x$, $U^0_P$ is smooth function satisfying (\ref{pos}), and
\[
\begin{split}
M_1\big(1-\frac{U^0_P}{u^0_e(0)}\big)\sqrt{-\ln\mu\big(1-\frac{U^0_P}{u^0_e(0)}\big)}\leqslant\frac{\p_y U^0_{P}}{u^0_e(0)}\leqslant M_2\big(1-\frac{U^0_P}{u^0_e(0)}\big)\sqrt{-\ln\mu\big(1-\frac{U^0_P}{u^0_e(0)}\big)},
\end{split}
\]
\[
\begin{split}
-M_3\sqrt{-\ln\mu\big(1-\frac{U^0_P}{u^0_e(0)}\big)}\leqslant\frac{\p^2_y U^0_{P}}{\p_y U^0_{P}}\leqslant-M_4\sqrt{-\ln\mu\big(1-\frac{U^0_P}{u^0_e(0)}\big)},
\end{split}
\]
\[
\begin{split}
|\frac{\p^3_y U^0_P\p_y U^0_P-(\p^2_y U^0_P)^2}{(\p_y U^0_P)^2}|\leqslant M_5,
\end{split}
\]
where $\mu$, $M_i$ are positive constants and $0<\mu<1$, moreover, $U^0_P$ satisfies the high order parabolic compatibility conditions and the high order derivatives of $U^0_P$ are decaying fast enough, then for any $L>0$, equations (\ref{pra them}) admits a smooth solution $[u^0_p,v^0_p]$ satisfying (\ref{pra them2}), and
\begin{align}
\|\z y\y^M\nabla^k(u^0_p(x,y)-u^0_e(0))\|_{\infty}\lesssim1, \hspace{2mm}\text{  for  }\hspace{2mm} 0\leqslant k \leqslant K,
\end{align}
here $K$ and $M$ are large constants.
\end{Proposition}

Notice that  $$u^0_p(x,y)-u^0_e(0)\sim exp(-\alpha y^2),$$ and the Blasius's self-similar solution is in this class. We can also deal with the case that $u^0_p(x,y)-u^0_e(0)\sim exp(-\alpha y)$.

After we solved Prandtl's equation (\ref{Prandtl}), we let $$u^0_b(x,y)=u^0_p(x,y)-u^0_e(x,0),$$ $$v^0_b(x,y)=\int_y^\infty u^0_{bx}(x,y')dy'$$ and $p^0_b=0.$

\subsection{The high-order profiles}
By the method of asymptotic matching expansions, we can deduce the equations of $[u^j_e,v^j_e]$ and $[u^j_b,v^j_b]$, $j=1,2.$
The first order Euler profile $[u^1_e,v^1_e,p^1_e]$ solves the linearized Euler equations around $[u^0_e,v^0_e]$:
\begin{align} \label{euler1}
\left\{
\begin{aligned}
&u^0_eu^1_{eX}+u^0_{eX}u^1_e+v^0_eu^1_{eY}+u^0_{eY}v^1_e+p^1_{eX}=0,\\
&u^0_ev^1_{eX}+v^0_{eX}u^1_e+v^0_ev^1_{eY}+v^0_{eY}v^1_e+p^1_{eY}=0,\\
&\p_X u^1_e + \p_Y v^1_e = 0,\\
&v^1_e|_{Y=0}=-v^0_b|_{y=0}. \\
\end{aligned}
\right.
\end{align}
Follow the idea of Iyer \cite{I}, we introduce new independent variables by
\begin{align}\label{vari change}
\begin{aligned}
\theta(X,Y)=X, \hspace{5 mm} \psi(X,Y)=\int_0^Yu^0_e(X,Y')dY'.
\end{aligned}
\end{align}
Let $\psi^1$ be the stream function of $[u^1_e,v^1_e]$
\[
\begin{split}
     \psi^1(X,Y):=\int_0^Yu^1_e(X,Y')dY'-\int_0^Xv^1_e(X',0)dX',\hspace{5 mm} \psi^1_Y=u^1_e, \hspace{5 mm} \psi^1_X=-v^1_e,
\end{split}
\]
then first Euler layer equations (\ref{euler1}) are equivalent to
\begin{align} \label{stream-euler1}
\begin{aligned}
     \p_\theta[\Delta_{XY} \psi^1-F_e'(\psi)\psi^1]=0.
\end{aligned}
\end{align}
Which reduce to find a solution of the following equations
\begin{align} \label{stream-euler1-bry}
\left\{
\begin{aligned}
     &\Delta \psi^1-F_e'(\psi)\psi^1=0,\\
     &\psi^1|_{X=0}=\psi^1_0(Y),\hspace{3 mm}\psi^1|_{X=L}=\psi^1_L(Y),\\
     &\psi^1|_{Y=0}=\int_0^Xv^0_b(X',0)dX',\hspace{3 mm}\psi^1|_{Y\rightarrow\infty}=0.
\end{aligned}
\right.
\end{align}
It is a standard elliptic problem, we have the following result.
\begin{Lemma}
If $v^0_b$ is a smooth functions, for any $L>0$, there exist some $\psi^1_0(Y)$, $\psi^1_L(Y)$ such that equations (\ref{stream-euler1-bry}) admit a smooth solution satisfying the following estimate
\begin{align}
\|\z Y\y^M\nabla^k\psi^1\|\lesssim1,\hspace{1mm}\text{  for  } 1\leqslant k \leqslant K, K \text{ and } M \text{ are large constants }.
\end{align}
\end{Lemma}
\begin{Remark}
The boundary conditions of $\psi^1$ in (\ref{stream-euler1-bry}) imply the following boundary conditions of $[u^1_e,v^1_e]$
\begin{align}\label{bry of euler1}
\begin{aligned}
&u^1_e|_{X=0}=\p_Y\psi^1_0(Y),\hspace{3 mm}u^1_e|_{X=L}=\p_Y\psi^1_L(Y),\\
&v^1_e|_{Y=0}=-v^0_b(X,0),\hspace{3 mm}[u^1_e,v^1_e]|_{Y\rightarrow\infty}=0.
\end{aligned}
\end{align}
So we only need to construct a solution $[u^1_e,v^1_e]$ to equations (\ref{euler1}) with boundary conditions (\ref{bry of euler1}).
\end{Remark}

Next we need to solve the first order boundary layer profile. For simplicity, we introduce some notations
\begin{equation}
\begin{aligned}
&u^k_p:=u^k_b+\sum_{j=0}^k\frac{y^j}{j!}\p^j_Yu^{k-j}_e|_{Y=0},\hspace{5 mm} u^{(k)}_e:=\sum_{j=0}^k\frac{y^j}{j!}\p^j_Yu^{k-j}_e|_{Y=0},\\
&v^k_p:=v^k_b-v^k_b|_{y=0}+\sum_{j=0}^{k}\frac{y^{j+1}}{(j+1)!}\p^{j+1}_Yv^{k-j}_e|_{Y=0},\hspace{5 mm} v^{(k)}_e:=\sum_{j=0}^{k}\frac{y^{j+1}}{(j+1)!}\p^{j+1}_Yv^{k-j}_e|_{Y=0}.
\end{aligned}
\end{equation}
And $[u^1_b,v^1_b,p^1_b]$ solves the linearized Prandtl's equations around $[u^0_p,v^0_p]$:
\begin{align} \label{prandtl1}
\left
\{
\begin{aligned}
& u^0_p \p_x u^1_b + u^1_b \p_x u^0_p + \p_y u^0_p [v^1_b - v^1_b|_{y = 0}] + v^0_p \p_y u^1_b- \p_{yy} u^1_b + \p_x p^1_b = f^{(1)}, \\
& \p_y p^1_b = 0,\\
& \p_x u^1_b + \p_y v^1_b = 0,\\
& u^1_b|_{y = 0} = -u^1_e|_{Y = 0}, \hspace{5 mm} [u^1_b, v^1_b]|_{y \rightarrow \infty} = 0,
\end{aligned}
\right.
\end{align}
where
\begin{equation}
\begin{aligned}
f^{(1)}=&-\{u^0_bu^{(1)}_{ex}+u^0_{bx}u^{(1)}_e+v^0_b\p_yu^{(1)}_e+u^0_{by}v^{(1)}_e\}.
\end{aligned}
\end{equation}
We see that $f^{(1)}$ decays fast when $y\rightarrow\infty$ from lemma \ref{pra0}. Since that above equations are linear parabolic type equations, we add a boundary condition on $u^1_b|_{x=0}$
\begin{align} \label{prandtl1 bry}
\left
\{
\begin{aligned}
& u^0_p \p_x u^1_b + u^1_b \p_x u^0_p  + v^0_p \p_y u^1_b+ [v^1_b - v^1_b|_{y = 0}]\p_y u^0_p - \p_{yy} u^1_b+ \p_x p^1_b = f^{(1)}, \\
& \p_y p^1_b = 0,\\
& \p_x u^1_b + \p_y v^1_b = 0,\\
& u^1_b|_{x = 0}=U^1_B, \hspace{5 mm} u^1_b|_{y = 0} = -u^1_e|_{Y = 0}, \hspace{5 mm} [u^1_b, v^1_b]|_{y \rightarrow \infty} = 0.
\end{aligned}
\right.
\end{align}
We can also discuss the compatibility conditions like Prandtl's equations. In our case, $v^0_p$ is different in \cite{GI-H}, because $v^0_p\sim yv^0_{eY}(x,0)$ as $y$ goes to $\infty$, still we have
\begin{Lemma}\label{pra1}
If $f^{(1)}$ and its derivatives are bounded and decaying rapidly, they satisfy the parabolic compatibility conditions, then equations (\ref{prandtl1 bry}) admit a unique solution $[u^1_b,v^1_b]$, and
\begin{align}
\| \z y\y^M\nabla^k u^1_b\|_\infty+\| \z y\y^M\nabla^k v^1_b\|_\infty \lesssim 1  \hspace{2mm}\text{  for  }\hspace{2mm} 0\leqslant k \leqslant K,
\end{align}
here $K$ and $M$ are large constants.
\end{Lemma}
We will give the proof of lemma \ref{pra1} in Appendix.

The second order Euler profile $[u^2_e,v^2_e,p^2_e]$ solves the linearized Euler equations around $[u^0_e,v^0_e]$ with the force terms:
\begin{align} \label{euler2}
\left\{
\begin{aligned}
&u^0_eu^2_{eX}+u^0_{eX}u^2_e+v^0_eu^2_{eY}+u^0_{eY}v^2_e+p^2_{eX}=F^{(2)},\\
&u^0_ev^2_{eX}+v^0_{eX}u^2_e+v^0_ev^2_{eY}+v^0_{eY}v^2_e+p^2_{eY}=G^{(2)},\\
&\p_X u^2_e + \p_Y v^2_e = 0,\\
&v^2_e|_{Y=0}=-v^1_b|_{y=0},\\
\end{aligned}
\right.
\end{align}
where
\begin{equation}
\begin{aligned}
&F^{(2)}=-(u^{1}_eu^1_{ex}+v^{1}_eu^1_{eY})+\Delta u^{0}_e,\\
&G^{(2)}=-(u^{1}_ev^1_{ex}+v^{1}_ev^1_{eY})+\Delta v^{0}_e.
\end{aligned}
\end{equation}
We can treat above equations as that of the first order Euler flow.
\begin{align} \label{stream-euler2}
\begin{aligned}
     \p_\theta[\Delta_{XY} \psi^2-F_e'(\psi)\psi^2-\frac{F''_e(\psi)}{2}(\psi^1)^2]=\frac{\Delta^2_{XY}\psi}{u}.
\end{aligned}
\end{align}
Let $$H(\theta,\psi)=\int_0^\theta \frac{\Delta^2_{XY}\psi(\theta',Y(\theta',\psi))}{u(\theta',\psi)}d\theta'.$$ Notice that $\psi\sim Y$ when $Y\rightarrow\infty$, we have that $H$ is of fast decay as $\psi\rightarrow\infty$ because of (\ref{Eul u}) and (\ref{Eul v}).
We can find a solution of the following equations
\begin{align} \label{stream-euler2-bry}
\left\{
\begin{aligned}
     &\Delta_{XY} \psi^2-F_e'(\psi)\psi^2-\frac{F''_e(\psi)}{2}(\psi^1)^2=H(\theta(X,Y),\psi(X,Y)),\\
     &\psi^2|_{X=0}=\psi^2_0(Y),\hspace{3 mm}\psi^2|_{X=L}=\psi^2_L(Y),\\
     &\psi^2|_{Y=0}=\int_0^Xv^1_b(X',0)dX',\hspace{3 mm}\psi^2|_{Y\rightarrow\infty}=0,
\end{aligned}
\right.
\end{align}
with suitable $\psi^2_0(Y)$, $\psi^2_L(Y)$, and we have estimates of the second order Euler flow
\begin{align}
\|\z Y\y^M\nabla^k\psi^2\|\lesssim1,\hspace{1mm}\text{  for  } 1\leqslant k \leqslant K,\quad K \text{ and } M \text{ large comstants }.
\end{align}

The second order boundary layer profile $[u^2_b,v^2_b,p^2_b]$ is similar to the first, we need to solve the following equations
\begin{align} \label{prandtl2 bry}
\left
\{
\begin{aligned}
& u^0_p \p_x u^2_b + u^2_b \p_x u^0_p + v^0_p \p_y u^2_b+  [v^2_b - v^2_b|_{y = 0}]\p_y u^0_p- \p_{yy} u^2_b + \p_x p^2_b = f^{(2)}, \\
& \p_y p^2_b = g^{(2)},\\
& \p_x u^2_b + \p_y v^2_b = 0,\\
& u^2_b|_{x = 0}=U^2_B, \hspace{5 mm} u^2_b|_{y = 0} = -u^2_e|_{Y = 0}, \hspace{5 mm} [u^2_b, v^2_b]|_{y \rightarrow \infty} = 0.
\end{aligned}
\right.
\end{align}
Where
\begin{equation}
\begin{aligned}
f^{(2)}=&-\{u^0_bu^{(2)}_{ex}+u^0_{bx}u^{(2)}_e+v^0_bu^{(2)}_{ey}+u^0_{by}v^{(2)}_e\\
        &+u^{1}_pu^1_{bx}+u^{1}_bu^{(1)}_{ex}+v^{1}_pu^1_{by}+v^{1}_bu^{(1)}_{ey}-u^{0}_{bxx}\},\\
g^{(2)}=&-\{u^{0}_bv^0_{px}+u^{(0)}_ev^0_{bx}+v^{0}_bv^0_{py}+(v^{(0)}_e+v^1_e|_{Y=0})v^0_{by}\\
        &-v^{0}_{byy}\}.\\
\end{aligned}
\end{equation}
We can see $f^{(2)}$ and $g^{(2)}$ decays fast when $y\rightarrow\infty$ from Lemma \ref{pra0} and Lemma \ref{pra1}. By using the same argument of Lemma \ref{pra1}, we have
\begin{align}
\| \z y\y^M\nabla^ku^2_b\|_\infty+\| \z y\y^M\nabla^k v^2_b \|_\infty \lesssim 1 \hspace{2mm}\text{  for  }\hspace{2mm} 0\leqslant k \leqslant K,
\end{align}
here $K$ and $M$ are large constants.

After that, $p_b^{3}$ is solved by
\begin{equation}\label{p3}
\begin{aligned}
p_b^{3}=&\int_y^\infty\{\sum_{j=0}^{1}[u^{1-j}_bv^j_{px}+u^{(1-j)}_ev^j_{bx}+v^{1-j}_bv^j_{py}\\
          &+(v^{(1-j)}_e+v^{1-j}_e|_{Y=0})v^j_{by}]-v^{1}_{byy}\}dy'.
\end{aligned}
\end{equation}
%%%%%
We can conclude the following proposition for the approximate profiles
\begin{Proposition}\label{construct} Under the assumptions of theorem \ref{main1} or theorem \ref{main2}, then equations (\ref{euler1}), (\ref{prandtl1}), (\ref{euler2}), (\ref{prandtl2 bry}) admit smooth solutions $[u^j_e,v^j_e]$, $[u^j_b,v^j_b]$ for $j=1,2$, and  the following estimates hold
\begin{align}
\begin{aligned}
&\| \z y\y^M\nabla^k u^j_b\|_\infty+\| \z y\y^M\nabla^k v^j_b\|_\infty \lesssim 1\hspace{2mm}\text{  for  }\hspace{2mm}  0\leqslant k \leqslant K,j=0,1,2, \\
&\| \z Y\y^M\nabla^ku^j_e\|_\infty+\| \z Y\y^M\nabla^kv^j_e\|_\infty  \lesssim 1 \hspace{2mm}\text{  for  }\hspace{2mm}  0\leqslant k \leqslant K,  j=1,2,
\end{aligned}
\end{align}
\noindent here $K,$ $M$ sufficiently large constants, $\z y\y=y+1$ and $\z Y\y=Y+1$.
\end{Proposition}

Notices that $v^2_b|_{y=0}\neq0$.  We need to match the boundary conditions at $y=0$, and also   $v^2_b|_{y \rightarrow\infty}=0$. Then we can modify $[\hat{u}^2_b,\hat{v}^2_b]$ in this way:
\begin{equation}\label{modify}
\begin{aligned}
&\hat{u}^2_b(x,y):=\chi(\sqrt\e y)u^2_b(x,y)-\sqrt\e\chi'(\sqrt\e y)\int_0^y u^2_b(x,y')dy',\\
&\hat{v}^2_b(x,y):=\chi(\sqrt\e y)(v^2_b(x,y)-v^2_b(x,0)),
\end{aligned}
\end{equation}
where $\chi$ is a cut-off function satisfying $\chi|_{[0,1]}=1$ and $\chi|_{[2,\infty]}=0$.

And let $[U_s,V_s,P_s]$ be
\begin{equation}
\begin{aligned}
U_s(X,Y)=&u^0_e(X,Y)+u^0_b(X,\frac{Y}{\se})+\se[u^1_e(X,Y)+u^1_b(X,\frac{Y}{\se})]\\
          &+\e[u^2_e(X,Y)+\hat{u}^2_b(X,\frac{Y}{\se})],\\
V_s(X,Y)=&v^0_e(X,Y)+\se[v^0_b(X,\frac{Y}{\se})+v^1_e(X,Y)]+\e[v^1_b(X,\frac{Y}{\se})+v^2_e(X,Y)]\\
          &+\e^\frac{3}{2}\hat{v}^2_b(X,\frac{Y}{\se}),\\
P_s(X,Y)=&p^0_e(X,Y)+p^0_b(X,\frac{Y}{\se})+\se[p^1_e(X,Y)+p^1_b(X,\frac{Y}{\se})]\\
          &+\e[p^2_e(X,Y)+p^2_b(X,\frac{Y}{\se})]+\e^\frac{3}{2}p^3_b(X,\frac{Y}{\se}).
\end{aligned}
\end{equation}
Then the errors
\begin{equation}\label{R1,R2}
\begin{aligned}
&R_1:=U_sU_{sX}+V_sU_{sY}-\e\Delta U_s+P_{sX},\\
&R_2:=U_sV_{sX}+V_sV_{sY}-\e\Delta V_s+P_{sY},
\end{aligned}
\end{equation}
satisfy
\begin{equation}
\begin{aligned}\label{R order}
\|R_1\|+\|R_2\|\lesssim\e^{\frac{3}{2}}.
\end{aligned}
\end{equation}

\begin{Remark}\label{Us}
We obtain what $[U_s,V_s]$ look like when $\e$ is small
\[
\begin{split}
&U_s(X,Y)= u^0_e(X,Y)+u^0_b(X,\frac{Y}{\se})+O\big(\se\big),\\
&V_s(X,Y)= v^0_e(X,Y)+\se (v^0_b(X,\frac{Y}{\se})+v^1_e(X,Y))+O\big(\e\big).
\end{split}
\]
When $\frac{Y}{\se}\leqslant1$,
\[
\begin{split}
u^0_e(X,Y)+u^0_b(X,\frac{Y}{\se})=u^0_e(X,Y)-u^0_e(X,0)+u^0_p(X,\frac{Y}{\se})\gtrsim-Y+\frac{Y}{\se}\gtrsim\frac{Y}{\se},
\end{split}
\]
when $1\leqslant\frac{Y}{\se}\leqslant \e^{-\frac{1}{4}}$,
\[
\begin{split}
u^0_e(X,Y)+u^0_b(X,\frac{Y}{\se})=u^0_e(X,Y)-u^0_e(X,0)+u^0_p(X,\frac{Y}{\se})\gtrsim-Y+1\gtrsim1,
\end{split}
\]
when $\frac{Y}{\se}\gtrsim \e^{-\frac{1}{4}}$,
\[
\begin{split}
u^0_e(X,Y)+u^0_b(X,\frac{Y}{\se})\gtrsim1,
\end{split}
\]
So $U_s\sim\frac{Y}{\se},$ when $Y\leqslant\se,$ and $U_s\sim 1,$ when $Y\geqslant\se.$

One can easily see for $i,j\geqslant0$,
\begin{align}
\begin{aligned}
&\|\p^j_XU_s\|_\infty\lesssim1,\hspace{2mm}\|\p^j_XV_s\|_\infty\lesssim1,\\
&\se^i\|Y^j\p^{i+j}_YU_s\|_\infty\lesssim\se^i\|Y^j\p^{i+j}_Yu^0_e\|_\infty+\se^i\|\frac{y^j}{\se^i}\p^{i+j}_yu^0_b\|_\infty+O\big(\se\big)\lesssim1,\\
&\se^i\|Y^j\p^{i+j+1}_YV_s\|_\infty\lesssim\se^i\|Y^j\p^{i+j+1}_Yu^0_e\|_\infty+\se^i\|\frac{y^j}{\se^i}\p^{i+j+1}_yu^0_b\|_\infty+O\big(\se\big)\lesssim1.
\end{aligned}
\end{align}
And in shear flow Euler case
\[
\begin{split}
U_{sX}(X,Y)\sim u^0_{bX}(X,\frac{Y}{\se}),\hspace{2mm} V_s(X,Y)\sim\se (v^0_b(X,\frac{Y}{\se})+v^1_e(X,Y)).
\end{split}
\]
\end{Remark}

\section{Estimates of the remainder}
Now we begin to estimate the remainder. Let
\begin{align}
\UE=U_s+V,\hspace{3 mm}\VE=V_s+V.
\end{align}
Then
\begin{equation}
\left\{
\begin{aligned}
&U_sU_X+U_{sX}U+V_sU_Y+U_{sY}V-\e\Delta U+P_X=-\{R_1+UU_X+VU_Y\},\\
&U_sV_X+V_{sX}U+V_sV_Y+V_{sY}V-\e\Delta V+P_Y=-\{R_2+UV_X+VV_Y\},\\
&U_X+V_Y=0.
\end{aligned}
\right.
\end{equation}
We consider the linearized equations:
\begin{equation}\label{LNSE}
\left\{
\begin{aligned}
&U_sU_X+U_{sX}U+V_sU_Y+U_{sY}V-\e\Delta U+P_X=F_1,\\
&U_sV_X+V_{sX}U+V_sV_Y+V_{sY}V-\e\Delta V+P_Y=F_2,\\
&U_X+V_Y=0.
\end{aligned}
\right.
\end{equation}
Our critical estimates is the following proposition
\begin{Proposition}\label{Prop}
Under the assumptions in theorem\ref{main1} or theorem\ref{main2}, if $$[U,V]\in W^{2,2}(\Omega)\cap W^{1,2}_0(\Omega)$$ satisfies the equations (\ref{LNSE}), \noindent then
\begin{align}
\|\se U_X,\se U_Y,\se V_X, \se V_Y, U, V\|\leqslant C(\|F_1\|+\|F_2\|).
\end{align}
\end{Proposition}
Let $\Phi$ be the stream-function of $U,V$, that is, $\Phi_X=-V, \Phi_Y=U$, then
\[
\left\{
\begin{aligned}
    &-\Delta\Phi=\partial_XV-\partial_YU,\\
    &\Phi|_{X=0}=\Phi|_{X=L}=\Phi|_{Y=0}=0.
\end{aligned}
\right.
\]
Because $[U,V]|_\Omega=0$ and $U_X+V_Y=0$,
$$
    \Phi_X|_{X=0}=\Phi_X|_{X=L}=\Phi_Y|_{Y=0}=0.
$$
We deduce the equation of stream function
\begin{equation}
\left\{
\begin{aligned}
&U_s\Delta\Phi_X-\Phi_X\Delta U_s-\e\Delta^2\Phi+V_s\Delta\Phi_Y-\Phi_Y\Delta V_s=\partial_Y F_1-\partial_X F_2,\\
&\Phi|_{X=0}=\Phi|_{X=L}=\Phi|_{Y=0}=\Phi_X|_{X=0}=\Phi_X|_{X=L}=\Phi_Y|_{Y=0}=0.
\end{aligned}
\right.
\end{equation}

Let $G=\frac{\Phi}{U_s}$, $G$ and $\Phi$ satisfies
\begin{equation}\label{G-Euq}
\left\{
\begin{aligned}
&\partial_{XX}[U_s^2G_X]+\partial_{XY}[U_s^2G_Y]-\e\Delta^2\Phi+R[\Phi]=\partial_Y F_1-\partial_X F_2,\\
&G|_{X=0}=G|_{X=L}=G|_{Y=0}=G_X|_{X=0}=G_X|_{X=L}=0.
\end{aligned}
\right.
\end{equation}
where $R[\Phi]=V_s\Delta\Phi_Y-U_{sX}\Delta\Phi-\Phi_Y\Delta V_s+\Phi\Delta U_{sX}$.
We define two norms of $G$
\begin{equation}\label{space XY}
\begin{aligned}
\|G\|_\X^2&:=\|U_sG_X\|^2+\|U_sG_Y\|^2+\e\{\|\sqrt U_sG_{XX}\sqrt\omega\|^2+2\|\sqrt U_sG_{XY}\sqrt\omega\|^2+\|\sqrt U_sG_{YY}\sqrt\omega\|^2\},\\
\|G\|_\Y^2&:=\e\{\|\sqrt U_sG_{XX}\|^2+2\|\sqrt U_sG_{XY}\|^2+\|\sqrt U_sG_{YY}\|^2\},
\end{aligned}
\end{equation}
where $\omega=L-X$.

Following the method of Guo and Iyer \cite{GI}, we have a Hardy-type's inequality: 
\begin{Lemma}\label{Hardy}
If $H\in W^{1,2}(0,\infty)$, $0<\xi\leqslant1$, \noindent then
$$ \|H\|_{L^2_Y}^2\leqslant C\xi\e\|\sqrt {U_s}H_Y\|_{L^2_Y}^2+\frac{C}{\xi^2}\|U_sH\|_{L^2_Y}^2.$$
\end{Lemma}
\begin{proof}
Let $\chi$ be a smooth cut-off function supported in $[0,2]$ ,and $\chi|_{[0,1]}=1$,
$$\int H^2\lesssim\int H^2\chi^2(\frac{Y}{\xi\se})+\int H^2(1-\chi(\frac{Y}{\xi\se}))^2.$$
Recall the leading profile of $U_s$, $U_s\sim\frac{Y}{\se},$ if $Y\leqslant\se,$ $U_s\sim 1,$ if $Y\geqslant\se.$ So when $\frac{Y}{\se}\leqslant1$, $1-\chi(\frac{Y}{\xi\se})\lesssim\frac{Y}{\xi\se}\lesssim\frac{U_s}{\xi}$, and when $\frac{Y}{\se}\geqslant1$, $1-\chi(\frac{Y}{\xi\se})\lesssim1\lesssim\frac{U_s}{\xi}$. We have
$$\int H^2(1-\chi(\frac{Y}{\xi\se}))^2\lesssim\frac{1}{\xi^2}\int U_s^2H^2,$$
and
\[
\begin{aligned}
\int H^2\chi^2(\frac{Y}{\xi\se})&=-2\int YHH_Y\chi^2(\frac{Y}{\xi\se})-2\int\frac{Y}{\xi\se}H^2\chi'(\frac{Y}{\xi\se})\chi(\frac{Y}{\xi\se})\\
                               &\lesssim\int Y^2H_Y^2\chi^2(\frac{Y}{\xi\se})+\int({\frac{Y}{\xi\se}})^2H^2|\chi'(\frac{Y}{\xi\se})|\chi(\frac{Y}{\xi\se})\\
                               &\lesssim\xi\e\int U_sH_Y^2+\frac{1}{\xi^2}\int U_s^2H^2.
\end{aligned}
\]
The proof is complete.
\end{proof}

\begin{Lemma}\label{routine}
Let $G$ is the solution of equation (\ref{G-Euq}), $L>0$, then
\begin{align}\label{rou}
\|G\|_\Y^2\lesssim\|G\|_\X^2+|\z\partial_YF_1-\partial_XF_2,g\y|.
\end{align}
\end{Lemma}
\begin{proof}
Taking the inner product of $(\ref{G-Euq})_1$ and $-G$.

First term
\begin{align}\label{a1}
\begin{aligned}
\z\partial_{XX}[U_s^2G_X],-G\y=&\z\partial_X[U_s^2G_X],G_X\y=\z U_sU_{sX}G_X,G_X\y\\
                                      =&O\big(\| U_sG_X\|^2\big)=O\big(\|G\|_\X^2\big).
\end{aligned}
\end{align}
Second term
\begin{align}\label{a2}
\begin{aligned}
\z\partial_{XY}[U_s^2G_Y],-G\y=&\z\partial_X[U_s^2G_Y],G_Y\y=\z U_sU_{sX}G_Y,G_Y\y\\
                              =&O\big(\|U_sG_Y\|^2\big)=O\big(\|G\|_\X^2\big).
\end{aligned}
\end{align}

Bi-Laplacian term
\[
\z-\e\Delta^2\Phi,-G\omega\y=\e\z\Phi_{XXXX}+2\Phi_{XXYY}+\Phi_{YYYY},G\omega\y.
\]
\[
\begin{split}
\e\z\Phi_{XXXX},G\y=&-\e\z\Phi_{XXX},G_X\y=\e\z\Phi_{XX},G_{XX}\y\\
                         =&\e\z U_sG_{XX}+2U_{sX}G_X+U_{sXX}G,G_{XX}\y\\
                         =&\e\z U_sG_{XX},G_{XX}\y-\e\z2U_{sXX}G_X+U_{sXXX}G,G_X\y,\\
                         =&\e\z U_sG_{XX},G_{XX}\y+O\big(\e\|G_X\|^2\big),
\end{split}
\]
 let $0<\xi\leqslant1$ be choosen latter, by lemma \ref{Hardy},
\[
\begin{split}
\|G_X\|^2\lesssim&\frac{1}{\xi^2}\|U_sG_X\|^2+\xi\e\|\sqrt{U_s}G_{XY}\|^2\\
    \lesssim&\frac{1}{\xi^2}\|G\|^2_{\X}+\xi\|G\|^2_\Y,
\end{split}
\]
then
\begin{align}\label{a3.1}
\begin{aligned}
\e\z\Phi_{XXXX},G\y=\e\z U_sG_{XX},G_{XX}\y+O\big(\frac{\e}{\xi^2}\|G\|^2_{\X}+\xi\e\|G\|^2_\Y\big).
\end{aligned}
\end{align}
Next
\[
\begin{split}
\z-2\e\Phi_{XXYY},-G\y=&-\z2\e\Phi_{XXY},G_Y\y=\z2\e\Phi_{XY},G_{XY}\y\\
                      =&2\e\z U_sG_{XY}+U_{sX}G_Y+U_{sY}G_X+U_{sXY}G,G_{XY}\y,\\
                      =&2\e\z U_sG_{XY},G_{XY}\y-\e\z U_{sXX}G_Y,G_Y\y-\e\z U_{sYY}G_X,G_X\y\\
      &-2\e\z U_{sXY}G_Y,G_X\y-2\e\z U_{sXYY}G,G_X\y\\
=&2\e\z U_sG_{XY},G_{XY}\y+O\big(\|G_X\|^2+\|G_Y\|^2\big),
\end{split}
\]
 by Lemma \ref{Hardy},
\[
\begin{split}
\|G_Y\|^2\lesssim&\frac{1}{\xi^2}\|U_sG_Y\|^2+\xi\e\|\sqrt{U_s}G_{YY}\|^2\\
    \lesssim&\frac{1}{\xi^2}\|G\|^2_{\X}+\xi\|G\|^2_\Y,
\end{split}
\]
then
\begin{align}\label{a3.2}
\begin{aligned}
\z-2\e\Phi_{XXYY},-G\y=2\e\z U_sG_{XY},G_{XY}\y+O\big(\frac{1}{\xi^2}\|G\|^2_{\X}+\xi\|G\|^2_\Y\big).
\end{aligned}
\end{align}
\begin{align}\label{a3.3}
\begin{aligned}
\z-\e\Phi_{YYYY},-G\y=&-\e\z\Phi_{YYY},G_Y\y=\e\z\Phi_{YY},G_Y\y|_{Y=0}+\e\z\Phi_{YY},G_{YY}\y\\
                     =&\e\z U_sG_{YY}+2U_{sY}G_Y+U_{sYY}G,G_Y\y|_{Y=0}\\
                      &+\e\z U_sG_{YY}+2U_{sY}G_Y+U_{sYY}G,G_{YY}\y\\
                     =&2\e\z U_{sY}G_Y,G_Y\y|_{Y=0}+\e\z U_sG_{YY},G_{YY}\y+\e\z2U_{sY}G_Y+U_{sYY}G,G_{YY}\y\\
                     =&\e\z U_{sY}G_Y,G_Y\y|_{Y=0}+\e\z U_sG_{YY},G_{YY}\y-\e\z2U_{sYY}G_Y+U_{sYYY}G,G_Y\y\\
   =&\e\z U_{sY}G_Y,G_Y\y|_{Y=0}+\e\z U_sG_{YY},G_{YY}\y\\
    &+O\big(\|G_Y\|^2+\e\|YG_{YYY}\|_\infty\|\frac{G}{Y}\|\|G_Y\|\big)\\
=&\e\z U_{sY}G_Y,G_Y\y|_{Y=0}+\e\z U_sG_{YY},G_{YY}\y+O\big(\|G_Y\|^2\big)\\
=&\e\z U_{sY}G_Y,G_Y\y|_{Y=0}+\e\z U_sG_{YY},G_{YY}\y+O\big(\frac{1}{\xi^2}\|G\|^2_{\X}+\xi\|G\|^2_\Y).
\end{aligned}
\end{align}
Since $\e\z U_{sY}G_Y,G_Y\y|_{Y=0}$ is positive, then $\e\z U_sG_{YY},G_{YY}\y$ is which we want. According to the fact
\[
\begin{split}
 &\|V_sU_{sY}\|_\infty\lesssim\|v^0_eU_{sY}\|_\infty+1\lesssim\|\frac{v^0_e}{Y}\|_\infty\|YU_{sY}\|_\infty+1\lesssim1,\\
  &\|\Phi_X\|=\|U_sG_X+U_{sX}G\|\lesssim\|G_X\|,\\
&\|\Phi_Y\|=\|U_sG_Y+U_{sY}G\|\lesssim\|G_Y\|+\|U_{sY}Y\|_{\infty}\|\frac{G}{Y}\|\lesssim\|G_Y\|.
\end{split}
\]
The $R[\Phi]$ term can be estimated as
\begin{align}\label{a4}
\begin{aligned}
\z V_s\Phi_{XXY},-G\y=&\z V_s\Phi_{XY},G_X\y+\z V_{sX}\Phi_{XY},G\y\\
                     =&\z V_s(U_sG_{XY}+U_{sX}G_Y+U_{sY}G_X+U_{sXY}G),G_X\y\\
                      &-\z V_{sX}\Phi_X,G_Y\y-\z V_{sXY}\Phi_X,G\y\\
                     =&-\frac{1}{2}\z(V_sU_s)_YG_X,G_X\y+\z V_s(U_{sX}G_Y+U_{sY}G_X+U_{sXY}G),G_X\y\\
                      &-\z V_{sX}\Phi_X,G_Y\y-\z V_{sXY}\Phi_X,G\y\\
                     =&O\big(\|G_X\|^2+\|G_Y\|^2\big)=O\big(\frac{1}{\xi^2}\|G\|^2_{\X}+\xi\|G\|^2_\Y).
\end{aligned}
\end{align}
\begin{align}\label{a5}
\begin{aligned}
\z V_s\Phi_{YYY},-G\y=&\z V_s\Phi_{YY},G_Y\y+\z V_{sY}\Phi_{YY},G\y\\
                     =&\z V_s(U_sG_{YY}+2U_{sY}G_Y+U_{sYY}G),G_Y\y\\
                       &-\z V_{sY}\Phi_Y,G_Y\y-\z V_{sYY}\Phi_Y,G\y\\
                     =&-\frac{1}{2}\z (V_sU_s)_YG_Y,G_Y\y+\z V_s(2U_{sY}G_Y+U_{sYY}G),G_Y\y\\
                      &-\z V_{sY}\Phi_Y,G_Y\omega\y-\z V_{sYY}\Phi_Y,G\y\\
                     =&O\big(\|G_Y\|^2\big)=O\big(\frac{1}{\xi^2}\|G\|^2_{\X}+\xi\|G\|^2_\Y).
\end{aligned}
\end{align}
\begin{align}\label{a6}
\begin{aligned}
\z -U_{sX}\Delta\Phi,-G\y=&-\z U_{sX}\Phi_X,G_X\y-\z U_{sXX}\Phi_X,G\y\\
               &-\z U_{sX}\Phi_Y,G_Y\y-\z U_{sXY}\Phi_Y,G\y\\
                           =&O\big(\|G_X\|^2+\|G_Y\|^2\big)=O\big(\frac{1}{\xi^2}\|G\|^2_{\X}+\xi\|G\|^2_\Y).
\end{aligned}
\end{align}
\begin{align}\label{a7}
\begin{aligned}
\z-\Phi_Y\Delta V_s+\Phi\Delta U_{sX},-G\y=&O\big(\|G_X\|^2+\|G_Y\|^2\big)\\
            =&O\big(\frac{1}{\xi^2}\|G\|^2_{\X}+\xi\|G\|^2_\Y).
\end{aligned}
\end{align}
Collecting (\ref{a1})-(\ref{a7}), choosing $\xi$ small enough, we can obtain the inequality (\ref{rou}).
\end{proof}

\begin{Lemma}\label{important1}
Let $G$ is the solution of equation (\ref{G-Euq}), $0<L,\|v^0_e\|_\infty\ll1$, then
\begin{align}
\|G\|_\X^2\lesssim (\sqrt\e+L+\|v^0_e\|_{\infty})\|G\|_\Y^2+|\z\partial_YF_1-\partial_XF_2,G\omega\y|.
\end{align}
\end{Lemma}
\begin{proof}
Taking the inner product of $(\ref{G-Euq})_1$ with $-G\omega$, where $\omega=L-x$.

Because $|U_{sX}|\lesssim U_s$ and $\omega\leqslant L$, the first term is
\begin{align}\label{b1}
\begin{aligned}
\z\partial_{XX}[U_s^2G_X],-G\omega\y=&-\z\partial_X[U_s^2G_X],G\y+\z\partial_X[U_s^2G_X],G_X\omega\y\\
                                    =&\frac{3}{2}\z U_s^2G_X,G_X\y+\z U_sU_{sX}G_X,G_X\omega\y\\
                                    =&\frac{3}{2}\z U_s^2G_X,G_X\y+O\big(L\|U_sG_X\|^2\big).
\end{aligned}
\end{align}
Second term
\begin{align}\label{b2}
\begin{aligned}
\z\partial_{XY}[U_s^2G_Y],-G\omega\y=&\z\partial_X[U_s^2G_Y],G_Y\omega\y\\
                                  =&\z U_s^2G_{XY},G_Y\omega\y+\z2U_sU_{sX}G_Y,G_Y\omega\y\\
                                  =&\frac{1}{2}\z U_s^2G_Y,G_Y\y+\z U_sU_{sX}G_Y,G_Y\omega\y\\
                =&\frac{1}{2}\z U_s^2G_Y,G_Y\y+O\big(L\|U_sG_Y\|^2\big).
\end{aligned}
\end{align}
Bi-Laplacian term
\[
\z-\e\Delta^2\Phi,-G\omega\y=\e\z\Phi_{XXXX}+2\Phi_{XXYY}+\Phi_{YYYY},G\omega\y.
\]
\begin{align}\label{b3.1}
\begin{aligned}
\e\z\Phi_{XXXX},G\omega\y=&-\e\z\Phi_{XXX},G_X\omega\y+\e\z\Phi_{XXX},G\y\\
                         =&\e\z\Phi_{XX},G_{XX}\omega\y-2\e\z\Phi_{XX},G_X\y\\
                         =&\e\z U_sG_{XX}+2U_{sX}G_X+U_{sXX}G,G_{XX}\omega\y\\
                          &-2\e\z U_sG_{XX}+2U_{sX}G_X+U_{sXX}G,G_X\y\\
                         =&\e\z U_sG_{XX},G_{XX}\omega\y-\e\z 2U_{sX}G_X+U_{sXX}G,G_X\y\\
                          &-\e\z2U_{sXX}G_X+U_{sXXX}G,G_X\omega\y\\
                =&\e\z U_sG_{XX},G_{XX}\omega\y+O\big(\e\|G_X\|\big)\\
               =&\e\z U_sG_{XX},G_{XX}\omega\y+O\big(\frac{\e}{\xi^2}\|G\|^2_{\X}+\xi\e\|G\|^2_\Y).
\end{aligned}
\end{align}
Next
\[
\begin{split}
2\e\z\Phi_{XXYY},G\omega\y=&-2\e\z\Phi_{XXY},G_Y\omega\y=2\e\z\Phi_{XY},G_{XY}\omega\y-2\e\z\Phi_{XY},G_Y\y\\
                          =&2\e\z U_sG_{XY}+U_{sX}G_Y+U_{sY}G_X+U_{sXY}G,G_{XY}\omega\y\\
                           &-2\e\z U_sG_{XY}+U_{sX}G_Y+U_{sY}G_X+U_{sXY}G,G_Y\y.
\end{split}
\]
And $2\e\z U_sG_{XY},G_{XY}\omega\y$ is good,
\[
\begin{split}
  &2\e\z U_{sX}G_Y+U_{sY}G_X+U_{sXY}G,G_{XY}\omega\y\\
 =&-\e\z U_{sXX}G_Y,G_Y\omega\y+\e\z U_{sX}G_Y,G_Y\y-\e\z U_{sYY}G_X,G_X\omega\y\\
  &-2\e\z U_{sXY}G_Y,G_X\omega\y-2\e\z U_{sXYY}G,G_Y\omega\y\\
 =&O\big(\e\|G_Y\|^2+\e\|U_{sYY}\|_\infty\|G_X\sqrt\omega\|^2\\
  &+\e\|U_{sXY}\|_\infty\|G_X\|\|G_Y\|+\e\|YU_{sXYY}\|_\infty\|\frac{G}{Y}\|\|G\|\big)\\
 =&O\big(\|G_X\sqrt\omega\|^2+\se\|G_X\|\|G_Y\|+\se\|G_Y\|^2\big),
\end{split}
\]
by Lemma \ref{Hardy}, then
\[
\begin{split}
&\|G_Y\|^2\lesssim\xi\e\|\sqrt {U_s}G_{YY}\|^2+\frac{1}{\xi^2}\|U_sG_Y\|^2\lesssim\frac{1}{\xi^2}\|G\|^2_\X+\xi\|G\|^2_\Y\\
&\|G_X\|^2\lesssim\xi\e\|\sqrt {U_s}G_{XY}\|^2+\frac{1}{\xi^2}\| U_sG_X\|^2\lesssim\frac{1}{\xi^2}\|G\|^2_\X+\xi\|G\|^2_\Y\\
&\| G_X\sqrt\omega\|^2\lesssim\xi\e\|\sqrt { U_s} G_{XY}\sqrt\omega\|^2+\frac{1}{\xi^2}\| U_sG_X\sqrt\omega\|^2\lesssim(\xi+\frac{L}{\xi^2})\|G\|^2_\X,
\end{split}
\]
so we have
\[
2\e\z U_{sX}G_Y+U_{sY}G_X+U_{sXY}G,G_{XY}\omega\y=O\big((\frac{\sqrt\e+L}{\xi^2}+\xi)\|G\|_\X^2+\se\xi\|G\|_\Y^2\big),
\]
and
\[
\begin{split}
&-2\e\z U_sG_{XY}+U_{sX}G_Y+U_{sY}G_X+U_{sXY}G,G_Y\y\\
=&-\e\z U_{sX}G_Y,G_Y\y-2\e\z U_{sY}G_Y,G_X\y-2\e\z U_{sXY}G,G_Y\y\\
=&O\big(\e\|G_Y\|^2+\e\|U_{sY}\|_\infty\|G_X\|\|G_Y\|+\e\|U_{sXY}Y\|_\infty\|\frac{G}{Y}\|\|G_Y\|\big)\\
=&O\big(\e\|G_Y\|^2+\se\|G_X\|\|G_Y\|\big)\\
=&O\big(\frac{\se}{\xi^2}\|G\|_\X^2+\sqrt\e\xi\|G\|_\Y^2\big).
\end{split}
\]
Therefore
\begin{align}\label{b3.2}
\begin{aligned}
2\e\z\Phi_{XXYY},G\omega\y=&2\e\z U_sG_{XY},G_{XY}\omega\y+O\big((\frac{\sqrt\e+L}{\xi^2}+\xi)\|G\|_\X^2+\se\xi\|G\|_\Y^2\big).
\end{aligned}
\end{align}
Integrating by parts, we have
\[
\begin{split}
\e\z\Phi_{YYYY},G\omega\y=&-\e\z\Phi_{YYY},G_Y\omega\y=\e\z \Phi_{YY},G_Y\omega\y|_{Y=0}+\e\z\Phi_{YY},G_{YY}\omega\y\\
                         =&\e\z U_sG_{YY}+2U_{sY}G_Y+U_{sYY}G,G_Y\omega\y|_{Y=0}\\
                          &+\e\z U_sG_{YY}+2U_{sY}G_Y+U_{sYY}G,G_{YY}\omega\y\\
                         =&2\e\z U_{sY}G_Y,G_Y\omega\y|_{Y=0}+\e\z U_sG_{YY},G_{YY}\omega\y+\e\z2U_{sY}G_Y+U_{sYY}G,G_{YY}\omega\y\\
                         =&\e\z U_{sY}G_Y,G_Y\omega\y|_{Y=0}+\e\z U_sG_{YY},G_{YY}\omega\y-\e\z 2U_{sYY}G_Y+U_{sYYY}G,G_Y\omega\y.
\end{split}
\]
Because $U_{sY}|_{Y=0}>0$, the first two terms are positive above, and
\[
\begin{split}
-\e\z 2U_{sYY}G_Y+U_{sYYY}G,G_Y\omega\y=&O\big(\e\|U_{sYY}\|_\infty\|G_Y\sqrt\omega\|^2+\e\|U_{sYYY}Y\|_\infty\|\frac{G}{Y}\sqrt\omega\|\|G_Y\sqrt\omega\|\big)\\
=&O\big((\xi+\frac{L}{\xi^2})\|G\|^2_\X\big),
\end{split}
\]
then we obtain
\begin{align}\label{b3.3}
\begin{aligned}
\e\z\Phi_{YYYY},G\omega\y=&
\e\z U_{sY}G_Y,G_Y\omega\y|_{Y=0}+\e\z U_sG_{YY},G_{YY}\omega\y+O\big((\xi+\frac{L}{\xi^2})\|G\|^2_\X\big).
\end{aligned}
\end{align}
Finally, we deal with $R[\Phi]$ term. Since
\[
\begin{split}
\z V_s\Phi_{XXY},-G\omega\y=&\z V_s\phi_{XY},G_X\omega\y+\z V_{sX}\Phi_{XY},G\omega\y-\z V_s\Phi_{XY},G\y\\
                           =&\z V_s(U_sG_{XY}+U_{sX}G_Y+U_{sY}G_X+U_{sXY}G),G_X\omega\y\\
                            &-\z V_{sX}\Phi_X,G_Y\omega\y-\z V_{sXY}\Phi_X,G\omega\y+\z V_s\Phi_X,G_Y\y+\z V_{sY}\Phi_X,G\y\\
                           =&-\frac{1}{2}\z(V_sU_s)_YG_X,G_X\omega\y+\z V_s(U_{sX}G_Y+U_{sY}G_X+U_{sXY}G),G_X\omega\y\\
                            &-\z V_{sX}\Phi_X,G_Y\omega\y-\z V_{sXY}\Phi_X,G\omega\y+\z V_s\Phi_X,G_Y\y+\z V_{sY}\Phi_X,G\y.
\end{split}
\]
Notice that
\[
\begin{split}
&\|V_sU_{sY}\|_{\infty}\lesssim\frac{\|v^0_eu^0_{by}\|_{\infty}}{\se}+1\lesssim\|\frac{v^0_e}{Y}\|_{\infty}\|yu^0_{by}\|_{\infty}+1\lesssim1,\\
&\|V_s\|_{\infty}\lesssim\|v^0_e\|_{\infty}+\se,
\end{split}
\]
we obtain
\begin{align}\label{b4}
\begin{aligned}
\z V_s\Phi_{XXY},-G\omega\y=&O\big(\|G_X\sqrt\omega\|^2+\|G_Y\sqrt\omega\|^2+\se\|G_X\|\|G_Y\|+\|v^0_e\|_{\infty}\|G_X\|\|G_Y\|\big)\\
                           =&O\big((\frac{\sqrt\e+L+\|v^0_e\|_{\infty}}{\xi^2}+\xi)\|G\|_\X^2+\xi(\sqrt\e+\|v^0_e\|_{\infty})\|G\|_\Y^2\big).
\end{aligned}
\end{align}
Similarly,
\begin{align}\label{b5}
\begin{aligned}
\z V_s\Phi_{YYY},-G\omega\y=&\z V_s\Phi_{YY},G_Y\omega\y+\z V_{sY}\Phi_{YY},G\omega\y\\
                           =&\z V_s(U_sG_{YY}+2U_{sY}G_Y+U_{sYY}G),G_Y\omega\y\\
                            &-\z V_{sY}\Phi_Y,G_Y\omega\y-\z V_{sYY}\Phi_Y,G\omega\y\\
                           =&-\frac{1}{2}\z (V_sU_s)_YG_Y,G_Y\omega\y+\z V_s(2U_{sY}G_Y+U_{sYY}G),G_Y\omega\y\\
                            &-\z V_{sY}\Phi_Y,G_Y\omega\y-\z V_{sYY}\Phi_Y,G\omega\y\\
                           =&O\big(\|G_Y\sqrt\omega\|^2+\|\Phi_Y\sqrt\omega\|^2\big)\\
                           =&O\big((\xi+\frac{L}{\xi^2})\|G\|^2_\X\big).
\end{aligned}
\end{align}
\begin{align}\label{b6}
\begin{aligned}
\z -U_{sX}\Delta\Phi,-G\omega\y=&-\z U_{sX}\Phi_X,G_X\omega\y-\z U_{sXX}\Phi_X,G\omega\y+\z U_{sX}\Phi_X,G\y\\
                                &-\z U_{sX}\Phi_Y,G_Y\omega\y-\z U_{sXY}\Phi_Y,G\omega\y\\
                                =&O\big(\|G_X\sqrt\omega\|^2+\|G_Y\sqrt\omega\|^2\big)\\
                                =&O\big((\xi+\frac{L}{\xi^2})\|G\|^2_\X\big).
\end{aligned}
\end{align}
\begin{align}\label{b7}
\begin{aligned}
\z-\Phi_Y\Delta V_s+\Phi\Delta U_{sX},-G\omega\y=&O\big(\|G_X\sqrt\omega\|^2+\|G_Y\sqrt\omega\|^2\big)\\
                                                =&O\big((\xi+\frac{L}{\xi^2})\|G\|^2_\X\big).
\end{aligned}
\end{align}
Collecting (\ref{b1})-(\ref{b7}), we get
\[
\begin{split}
\|G\|^2_\X\lesssim& (\frac{\sqrt\e+L+\|v^0_e\|_{\infty}}{\xi^2}+\xi)\|G\|_\X^2+\xi(\sqrt\e+\|v^0_e\|_{\infty})\|G\|_\Y^2\\
                  &+|\z\partial_YF_1-\partial_XF_2,G\omega\y|.
\end{split}
\]
Finally we choose $\xi=(\se+L+\|v^0_e\|_\infty)^{\frac{1}{4}}$ to be a small constant, then we finish the proof.
\end{proof}

\begin{Lemma}\label{important2}
 Let $G$ is the solution of equation (\ref{G-Euq}), with $u^0_e$ satisfies (\ref{shear Eul profile}), $u^0_p$ satisfies (\ref{pra them2}), then
\begin{align}
\|G\|_\X^2\leqslant C\se\|G\|_\Y^2+\z\partial_YF_1-\partial_XF_2,G\omega\y.
\end{align}
\end{Lemma}
\begin{proof}
Taking the inner product of $(\ref{G-Euq})_1$ with $-G\omega$, where $\omega=L-x$.

First term
\begin{align}
\begin{aligned}\label{t1}
\z\partial_{XX}[U_s^2G_X],-G\omega\y&=-\z\partial_X[U_s^2G_X],G\y+\z\partial_X[U_s^2G_X],G_X\omega\y\\
                                    &=\frac{3}{2}\z U_s^2G_X,G_X\y+\z U_sU_{sX}G_X,G_X\omega\y.
\end{aligned}
\end{align}
Second term
\begin{align}
\begin{aligned}\label{t2}
\z\partial_{XY}[U_s^2G_Y],-G\omega\y&=\z\partial_X[U_s^2G_Y],G_Y\omega\y\\
                                  &=\z U_s^2G_{XY},G_Y\omega\y+\z2U_sU_{sX}G_Y,G_Y\omega\y\\
                                  &=\frac{1}{2}\z U_s^2G_Y,G_Y\y+\z U_sU_{sX}G_Y,G_Y\omega\y.
\end{aligned}
\end{align}
Bi-Laplacian term
\[
\z-\e\Delta^2\Phi,-G\omega\y=\e\z\Phi_{XXXX}+2\Phi_{XXYY}+\Phi_{YYYY},G\omega\y.
\]
\begin{align}
\begin{aligned}\label{t3.1}
\e\z\Phi_{XXXX},G\omega\y=&-\e\z\Phi_{XXX},G_X\omega\y+\e\z\Phi_{XXX},G\y\\
                         =&\e\z\Phi_{XX},G_{XX}\omega\y-2\e\z\Phi_{XX},G_X\y\\
                         =&\e\z U_sG_{XX}+2U_{sX}G_X+U_{sXX}G,G_{XX}\omega\y\\
                          &-2\e\z U_sG_{XX}+2U_{sX}G_X+U_{sXX}G,G_X\y\\
                         =&\e\z U_sG_{XX},G_{XX}\omega\y-\e\z 2U_{sX}G_X+U_{sXX}G,G_X\y\\
                          &-\e\z2U_{sXX}G_X+U_{sXXX}G,G_X\omega\y,\\
                         =&\e\z U_sG_{XX},G_{XX}\omega\y+O\big(\e\|G_X\|^2\big),\\
                         =&\e\z U_sG_{XX},G_{XX}\omega\y+O\big(\e\|G\|^2_{\X}+\e\|G\|^2_{\Y}\big),
\end{aligned}
\end{align}
next
\begin{align}
\begin{aligned}\label{t3.2}
2\e\z\Phi_{XXYY},G\omega\y=&-2\e\z\Phi_{XXY},G_Y\omega\y=2\e\z\Phi_{XY},G_{XY}\omega\y-2\e\z\Phi_{XY},G_Y\y\\
                          =&2\e\z U_sG_{XY}+U_{sX}G_Y+U_{sY}G_X+U_{sXY}G,G_{XY}\omega\y\\
                           &-2\e\z U_sG_{XY}+U_{sX}G_Y+U_{sY}G_X+U_{sXY}G,G_Y\y\\%
                          =&2\e\z U_sG_{YY},G_{YY}\omega\y-\e\z U_{sXX}G_Y,G_Y\omega\y+\e\z U_{sX}G_Y,G_Y\y\\
                           &-\e\z U_{sYY}G_X,G_X\omega\y-2\e\z U_{sXY}G_Y,G_X\omega\y-2\e\z U_{sXYY}G,G_Y\omega\y\\
                           &-\e\z U_{sX}G_Y,G_Y\y-2\e\z U_{sY}G_Y,G_X\y-2\e\z U_{sXY}G,G_Y\y\\
                          =&2\e\z U_sG_{YY},G_{YY}\omega\y-\e\z U_{sYY}G_X,G_X\omega\y+O\big(\se\|G_X\|^2+\se\|G_Y\|^2\big)\\
                          =&2\e\z U_sG_{YY},G_{YY}\omega\y-\e\z U_{sYY}G_X,G_X\omega\y+O\big(\se\|G\|^2_{\X}+\se\|G\|^2_{\Y}\big),
\end{aligned}
\end{align}
and
\begin{align}
\begin{aligned}\label{t3.3}
\e\z\Phi_{YYYY},G\omega\y=&-\e\z\Phi_{YYY},G_Y\omega\y=\e\z \Phi_{YY},G_Y\omega\y|_{Y=0}+\e\z\Phi_{YY},G_{YY}\omega\y\\
                         =&\e\z U_sG_{YY}+2U_{sY}G_Y+U_{sYY}G,G_Y\omega\y|_{Y=0}\\
                          &+\e\z U_sG_{YY}+2U_{sY}G_Y+U_{sYY}G,G_{YY}\omega\y\\
                         =&2\e\z U_{sY}G_Y,G_Y\omega\y|_{Y=0}+\e\z U_sG_{YY},G_{YY}\omega\y+\e\z2U_{sY}G_Y+U_{sYY}G,G_{YY}\omega\y\\
                         =&\e\z U_{sY}G_Y,G_Y\omega\y|_{Y=0}+\e\z U_sG_{YY},G_{YY}\omega\y-\e\z 2U_{sYY}G_Y+U_{sYYY}G,G_Y\omega\y,\\
                         =&\e\z U_{sY}G_Y,G_Y\omega\y|_{Y=0}+\e\z U_sG_{YY},G_{YY}\omega\y\\
                          &-2\e\z U_{sYY}G_Y,G_Y\omega\y+\frac{\e}{2}\z U_{sYYYY}G,G\omega\y,
\end{aligned}
\end{align}
because $U_{sY}|_{Y=0}>0$, the first term above is positive. Now, we begin to deal with $R[\Phi]$ term
\begin{align}
\begin{aligned}\label{t4.1}
\z V_s\Phi_{XXY},-G\omega\y=&\z V_s\phi_{XY},G_X\omega\y+\z V_{sX}\Phi_{XY},G\omega\y-\z V_s\Phi_{XY},G\y\\
                           =&\z V_s(U_sG_{XY}+U_{sX}G_Y+U_{sY}G_X+U_{sXY}G),G_X\omega\y\\
                            &-\z V_{sX}\Phi_X,G_Y\omega\y-\z V_{sXY}\Phi_X,G\omega\y+\z V_s\Phi_X,G_Y\y+\z V_{sY}\Phi_X,G\y\\
                           =&\z V_sU_sG_{XY},G_X\omega\y+\z V_sU_{sY}G_X,G_X\omega\y+O\big(\se\|G_X\|^2+\se\|G_Y\|^2\big)\\
                           =&\frac{1}{2}\z (V_sU_{sY}-V_{sY}U_s)G_X,G_X\omega\y+O\big(\se\|G\|^2_{\X}+\se\|G\|^2_{\Y}\big),
\end{aligned}
\end{align}
\begin{align}
\begin{aligned}\label{t4.2}
\z V_s\Phi_{YYY},-G\omega\y=&\z V_s\Phi_{YY},G_Y\omega\y+\z V_{sY}\Phi_{YY},G\omega\y\\
                           =&\z V_s(U_sG_{YY}+2U_{sY}G_Y+U_{sYY}G),G_Y\omega\y\\
                            &-\z V_{sY}\Phi_Y,G_Y\omega\y-\z V_{sYY}\Phi_Y,G\omega\y\\
                           =&-\frac{1}{2}\z (V_sU_s)_YG_Y,G_Y\omega\y+\z 2V_sU_{sY}G_Y,G_Y\omega\y-\frac{1}{2}\z (V_sU_{sYY})_{Y}G,G\omega\y\\
                            &-\z V_{sY}U_sG_Y,G_Y\omega\y-\z V_{sY}U_{sY}G,G_Y\omega\y\\
                            &-\z V_{sYY}U_sG_Y,G\omega\y-\z V_{sYY}U_{sY}G,G\omega\y\\
                           =&-\frac{1}{2}\z (V_sU_s)_YG_Y,G_Y\omega\y+\z 2V_sU_{sY}G_Y,G_Y\omega\y-\frac{1}{2}\z (V_sU_{sYY})_{Y}G,G\omega\y\\
                            &-\z V_{sY}U_sG_Y,G_Y\omega\y+\frac{1}{2}\z (V_{sY}U_{sY})_YG,G\omega\y\\
                            &+\frac{1}{2}\z (V_{sYY}U_s)_YG,G\omega\y-\z V_{sYY}U_{sY}G,G\omega\y\\
                           =&\frac{3}{2}\z (V_sU_{sY}-V_{sY}U_s)G_Y,G_Y\omega\y+\frac{1}{2}\z(-V_sU_{sYYY}+V_{sYYY}U_s)G,G\omega\y,
\end{aligned}
\end{align}
\begin{align}
\begin{aligned}\label{t5}
\z -U_{sX}\Delta\Phi,-G\omega\y=&-\z U_{sX}\Phi_X,G_X\omega\y-\z U_{sXX}\Phi_X,G\omega\y+\z U_{sX}\Phi_X,G\y\\
                                &-\z U_{sX}\Phi_Y,G_Y\omega\y-\z U_{sXY}\Phi_Y,G\omega\y\\
                               =&-\z U_{sX}\Phi_X,G_X\omega\y-\z U_{sX}\Phi_Y,G_Y\omega\y\\
                                &-\z U_{sXY}\Phi_Y,G\omega\y+O\big(\se\|G_X\|^2+\se\|G_Y\|^2\big)\\
                               =&-\z U_{sX}U_sG_X,G_X\omega\y-\z U_{sX}U_sG_Y,G_Y\omega\y-\z U_{sX}U_{sY}G,G_Y\omega\y\\
                                &-\z U_{sXY}U_sG_Y,G\omega\y-\z U_{sXY}U_{sY}G,G\omega\y+O\big(\se\|G_X\|^2+\se\|G_Y\|^2\big)\\
                               =&-\z U_{sX}U_sG_X,G_X\omega\y-\z U_{sX}U_sG_Y,G_Y\omega\y\\
                                &+\frac{1}{2}\z (U_{sX}U_{sY})_YG,G\omega\y+\frac{1}{2}\z (U_{sXY}U_s)_YG,G\omega\y\\
                                &-\z U_{sXY}U_{sY}G,G\omega\y+O\big(\se\|G_X\|^2+\se\|G_Y\|^2\big)\\
                               =&-\z U_{sX}U_sG_X,G_X\omega\y-\z U_{sX}U_sG_Y,G_Y\omega\y\\
                                &+\frac{1}{2}\z (U_{sX}U_{sYY}+U_{sXYY}U_s)G,G\omega\y+O\big(\se\|G\|^2_{\X}+\se\|G\|^2_{\Y}\big),
\end{aligned}
\end{align}
\begin{align}
\begin{aligned}\label{t6}
\z-\Phi_Y\Delta V_s+\Phi\Delta U_{sX},-G\omega\y=&\z\Phi_YV_{sYY},G\omega\y-\z\Phi U_{sXYY},G\omega\y\\
         =&+\z\Phi_YV_{sXX},G\omega\y-\z\Phi U_{sXXX},G\omega\y\\
         =&\z\Phi_YV_{sYY},G\omega\y-\z\Phi U_{sYY},G\omega\y\\
          &+O\big(\se\|G_X\|^2+\se\|G_Y\|^2\big)\\
         =&\z U_sV_{sYY}G_Y,G\omega\y+\z U_{sY}V_{sYY}G,G\omega\y\\
          &-\z U_sU_{sXYY},G\omega\y+O\big(\se\|G_X\|^2+\se\|G_Y\|^2\big)\\
         =&-\frac{1}{2}\z (U_sV_{sYY})_YG,G\omega\y+\z U_{sY}V_{sYY}G,G\omega\y\\
          &-\z U_sU_{sXYY},G\omega\y+O\big(\se\|G_X\|^2+\se\|G_Y\|^2\big)\\
         =&\frac{1}{2}\z (U_{sY}V_{sYY}-U_sU_{sXYY})G,G\omega\y+O\big(\se\|G\|^2_{\X}+\se\|G\|^2_{\Y}\big).
\end{aligned}
\end{align}
Collecting (\ref{t1})-(\ref{t6}), we have
\begin{align}
\begin{aligned}\label{all}
  &\z\partial_{XX}[U_s^2G_X]+\partial_{XY}[U_s^2G_Y]-\e\Delta^2\Phi+R[\Phi],-G\omega\y\\
=&\frac{3}{2}\|U_sG_X\|^2+\frac{1}{2}\|U_sG_Y\|^2+\e\z U_{sY}G_Y,G_Y\omega\y|_{Y=0}\\
&+\e\{\|\sqrt U_sG_{XX}\sqrt\omega\|^2+2\|\sqrt U_sG_{XY}\sqrt\omega\|^2+\|\sqrt U_sG_{YY}\sqrt\omega\|^2\}\\
&+\z(-\e U_{sYY}+\frac{1}{2}V_sU_{sY}-\frac{1}{2}V_{sY}U_s)G_X,G_X\omega\y\\
&+\z(-2\e U_{sYY}+\frac{3}{2}V_sU_{sY}-\frac{3}{2}V_{sY}U_s)G_Y,G_Y\omega\y\\
&+\frac{1}{2}\z (\e U_{sYYYY}-U_sU_{sXYY}-V_sU_{sYYY}+U_{sX}U_{sYY}+U_{sY}V_{sYY})G,G\omega\y\\
&+O\big(\se\|G\|^2_{\X}+\se\|G\|^2_{\Y}\big).
\end{aligned}
\end{align}
Notice that
\[
\begin{split}
-\e U_{sYY}+V_sU_{sY}-V_{sY}U_s=&-\e U_{sYY}+V_sU_{sY}+U_sU_{sX}\\
   =&-u^0_{pyy}+v^0_p u^0_{py}+u^0_p u^0_{px}+O\big(\se\big)=\big(\se\big),
\end{split}
\]
and
\[
\begin{split}
-\e U_{sYY}=-u^0_{pyy}+O\big(\se\big)\geqslant O\big(\se\big),
\end{split}
\]
by assumptions (\ref{pra them2}). Then
\begin{align}
\begin{aligned}\label{1order}
&\z(-\e U_{sYY}+\frac{1}{2}V_sU_{sY}-\frac{1}{2}V_{sY}U_s)G_X,G_X\omega\y\\
&+\z(-2\e U_{sYY}+\frac{3}{2}V_sU_{sY}-\frac{3}{2}V_{sY}U_s)G_Y,G_Y\omega\y\\
&\geqslant O\big(\se\|G_X\|^2+\se\|G_Y\|^2\big).
\end{aligned}
\end{align}
Notice that
\[
\begin{split}
&Y^2(\e U_{sYYYY}-U_sU_{sXYY}-V_sU_{sYYY}+U_{sX}U_{sYY}+U_{sY}V_{sYY})\\
=&Y^2(\e U_{sYY}-U_sU_{sX}-V_sU_{sY})_{YY}\\
=&y^2(u^0_{pyy}-u^0_p u^0_{px}-v^0_p u^0_{py})_{yy}+O\big(\se\big)=O\big(\se\big).
\end{split}
\]
So
\begin{align}
\begin{aligned}\label{0order}
&\frac{1}{2}\z (\e U_{sYYYY}-U_sU_{sXYY}-V_sU_{sYYY}+U_{sX}U_{YY}+U_{sY}V_{YY})G,G\omega\y\\
=&O\big(\se\|\frac{G}{Y}\|^2\big)=O\big(\se\|G_Y\|^2\big)=O\big(\se\|G\|^2_{\X}+\se\|G\|^2_{\Y}\big).
\end{aligned}
\end{align}
Combining (\ref{all}), (\ref{1order}) and (\ref{0order}), we obtain
\begin{align}
\begin{aligned}\label{end}
  &\z\partial_{XX}[U_s^2G_X]+\partial_{XY}[U_s^2G_Y]-\e\Delta^2\Phi+R[\Phi],-G\omega\y\\
\geqslant&\frac{3}{2}\|U_sG_X\|^2+\frac{1}{2}\|U_sG_Y\|^2+\e\z U_{sY}G_Y,G_Y\omega\y|_{Y=0}\\
&+\e\{\|\sqrt U_sG_{XX}\sqrt\omega\|^2+2\|\sqrt U_sG_{XY}\sqrt\omega\|^2+\|\sqrt U_sG_{YY}\sqrt\omega\|^2\}\\
&+O\big(\se\|G\|^2_{\X}+\se\|G\|^2_{\Y}\big).
\end{aligned}
\end{align}
So we end the proof.
\end{proof}

{\bf Proof of Proposition \ref{Prop}:}\\
Under the assumptions of Theorem \ref{main1}, by Lemma \ref{routine} and Lemma \ref{important1}, we have
\[
\begin{split}
\|G\|_\Y^2&\lesssim\|G\|_\X^2+|\z\partial_YF_1-\partial_XF_2,G\y|,\\
\|G\|_\X^2&\lesssim(\sqrt\e+L+\|v^0_e\|_{\infty})\|G\|_\Y^2+|\z\partial_YF_1-\partial_XF_2,G\omega\y|,
\end{split}
\]
similarly, under the assumptions of Theorem \ref{main2}, by Lemma \ref{routine} and Lemma \ref{important2}, we have
\[
\begin{split}
\|G\|_\Y^2&\lesssim\|G\|_\X^2+|\z\partial_YF_1-\partial_XF_2,G\y|,\\
\|G\|_\X^2&\lesssim\se\|G\|_\Y^2+|\z\partial_YF_1-\partial_XF_2,G\omega\y|.
\end{split}
\]
Let $\delta=\sqrt\e+L+\|v^0_e\|_{\infty}$ in the first case, $\delta=\se$ in the second case, $\delta$ is small, then we have
\[
\begin{split}
\|G\|_\X^2+\|G\|_\Y^2&\lesssim\|G\|_\X^2+|\z\partial_YF_1-\partial_XF_2,G\y|\\
                                    &\lesssim\delta\|G\|_\Y^2+\z\partial_YF_1-\partial_XF_2,G\omega\y|+|\z\partial_YF_1-\partial_XF_2,G\y|\\
                                    &\lesssim(\|F_1\|+\|F_2\|)(\|G_Y\|+\|G_X\|)\\
                                    &\lesssim(\|F_1\|+\|F_2\|)(\|G\|_\X+\|G\|_\Y).
\end{split}
\]
It is easy to see
\[
\|\se\Phi_{XX},\se\Phi_{XY},\se\Phi_{YY},\Phi_X,\Phi_Y\|\lesssim\|G\|_\X+\|G\|_\Y\lesssim\|F_1\|+\|F_2\|.
\]
Then we obtain the proof.

\section{Proof of the main Theorems}
{\bf Proof of Theorem \ref{main1} and Theorem \ref{main2}:}
Let $\Us=[U_s,V_s]$, $\U=[U,V]$ and $\RE=[R_1,R_2]$, now we write Navier-Stokes equation in the following form
\begin{equation}\label{U-NS}
\left\{
\begin{aligned}
&-\e\Delta \U+\Us\cdot\nabla\U+\U\cdot\nabla\Us+\U\cdot\nabla \U+\nabla P=-\RE, \\
&\nabla \cdot \U=0,\hspace{5mm} \U|_\Omega=0.
\end{aligned}
\right.
\end{equation}
We use the method of contraction mapping. Define
\[
\begin{split}
\|\U\|_\Z:=\|\U\|+\se\|\nabla\U\|+\e^\frac{3}{2}\|\nabla^2\U\|.
\end{split}
\]
We denote $\T:W^{2,2}(\Omega)\rightarrow W^{2,2}(\Omega)$ as this way, $\T(\U)=\W$ where $\W$ is given by
\begin{equation}
\left\{
\begin{aligned}
&-\e\Delta \W+\Us\cdot\nabla\W+\W\cdot\nabla\Us+\nabla P=-\RE-\U\cdot\nabla\U, \\
&\nabla \cdot \W=0,\hspace{5mm} \W|_\Omega=0.
\end{aligned}
\right.
\end{equation}
Let $B=\{\U\in W^{2,2}(\Omega):\|\U\|_\Z\leqslant C_0(L)\e^{\frac{3}{2}}\}$, $C_0$ is chosen latter. Next we prove $\T$ is a contractive mapping in $B$, if $\|\RE\|\leqslant C_1\e^{\frac{3}{2}}$. We write $\F=-\RE-\U\cdot\nabla \U$, from Proposition \ref{Prop},
\[
\begin{split}
\|\W\|+\se\|\nabla\W\|\lesssim\|\F\|.
\end{split}
\]
Due to the $W^{2,2}$ estimate of Stokes equations in convex polygon in \cite{KE},
\[
\begin{split}
\e\|\nabla^2\W\|\lesssim_L\|\F\|+\|\nabla\W\|+\frac{1}{\se}\|\W\|\lesssim_L\frac{1}{\se}\|\F\|.
\end{split}
\]
So we get
\[
\begin{split}
\|\W\|_\Z\leqslant C_2(L)\|\F\|.
\end{split}
\]
It's easy to see
\[
\begin{split}
\|\U\cdot\nabla \U\|\lesssim_L\|\U\|_{L^\infty}\|\nabla\U\|\lesssim_L\|\U\|^{\frac{1}{4}}\|\nabla\U\|^{\frac{3}{2}}\|\nabla^2\U\|^{\frac{1}{4}}\lesssim_L\e^{-\frac{9}{8}}\|\U\|^2_\Z.
\end{split}
\]
It implies
\[
\begin{split}
\|\W\|_\Z\leqslant C_2(L)\|\F\|+C_3(L)\e^{-\frac{9}{8}}\|\U\|^2_\Z\leqslant (C_1C_2+C_3C^2_0\e^\frac{3}{8})\e^{\frac{3}{2}}.
\end{split}
\]
Select $C_0(L)=C_1(L)C_2(L)+1$, $\T(B)\subset B$ when $\e$ is small enough. And if $\U_1, \U_2\in B$,
\[
\begin{split}
\|\T(\U_1-\U_2)\|_\Z&\leqslant C_2(L)\|\U_1\cdot\nabla\U_1-\U_2\cdot\nabla\U_2\|\\
                    &\leqslant C_2(L)\|(\U_1-\U_2)\cdot\nabla\U_1\|+\|\U_2\cdot\nabla(\U_1-\U_2)\|\\
                    &\leqslant C_2(L)\|(\U_1-\U_2)\|_\infty\|\nabla\U_1\|+\|\U_2\|_\infty\|\nabla(\U_1-\U_2)\|\\
                    &\leqslant C_3(L)\e^{-\frac{9}{8}}(\|\U_1\|_\Z+\|\U_2\|_\Z)\|\U_1-\U_2\|_\Z\\
                    &\leqslant 2C_1(L)C_3(L)\e^\frac{3}{8}\|\U_1-\U_2\|_\Z,
\end{split}
\]
so $\T$ is a contraction mapping on $B$ when $\e$ is small enough, we can conclude equations (\ref{U-NS}) admits a unique solution and
\[
\begin{split}
\|\U\|_{L^\infty}\lesssim_L\e^{-\frac{5}{8}}\|\U\|_\Z\lesssim_L\e^{\frac{7}{8}}.
\end{split}
\]
So we have
\[
\begin{split}
 &|U^\e(X,Y)-u^0_e(X,Y)-u^0_b(X,\frac{Y}{\se})|\\
=&|\se u^1_e(X,Y)+\se u^1_b(X,\frac{Y}{\se})+\e u^2_e(X,Y)+\e\hat u^2_b(X,\frac{Y}{\se})+U(X,Y)|\\
\lesssim_L&\se,\\
 &|V^\e(X,Y)-v^0_e(X,Y)|\\
=&|\se v^0_b(X,\frac{Y}{\se})+\se v^1_e(X,Y)+\e v^1_b(X,\frac{Y}{\se})+\e v^2_e(X,Y)+\e^\frac{3}{2}\hat{v}^2_b(X,\frac{Y}{\se})+V(X,Y)|\\
\lesssim_L&\se,
\end{split}
\]
which ends the proof.

\part*{Appendix}

In this appendix, we prove lemma \ref{pra1}. For convenience, we write $[\ub,\vb]:=[u^0_p,v^0_p]$ and we homogenize the system (\ref{prandtl1 bry}) like \cite{GI-H}:
\begin{align}
\begin{aligned}
&u(x,y)=u^1_b(x,y)+u^1_e(x,0)\eta(y),\\
&v(x,y)=v^1_b(x,y)-v^1_b(x,0)+u^1_{eX}(x,0)I_\eta(y), \\
&I_\eta(y):=\int_y^\infty\eta(y')dy'.
\end{aligned}
\end{align}
Here, we select $\eta$ to be a $C^\infty$ function satisfying the following:
\begin{align}
\eta(0)=1,\hspace{5mm} \int_0^\infty\eta=0,\hspace{5mm} \eta\text{ decays fast as }y\rightarrow\infty.
\end{align}
Due to (\ref{prandtl1 bry}), the homogenized unknowns $[u,v]$ satisfy the system
\begin{align}
\left
\{
\begin{aligned}
& \ub \p_x u + u \p_x \ub + \vb \p_y u + v\p_y \ub - \p_{yy} u +p_x= f^{(1)}+F=: h, \\
& p_y=0,\\
& \p_x u + \p_y v = 0,\\
& u|_{x = 0}=U^1_B+u^1_e(0,0)\eta(y)=:u_0(y), \hspace{3 mm} [u,v]|_{y = 0} =0, \hspace{3 mm} u|_{y \rightarrow \infty} = 0,
\end{aligned}
\right.
\end{align}
where
\begin{align}
F=\ub u^1_{eX}(x,0)\eta+\ub_xu^1_e(x,0)\eta+\vb u^1_e(x,0)\eta'+\ub_y u^1_{eX}(x,0)I_\eta-u^1_e(x,0)\eta''.
\end{align}
Notice that $p$ is independent on $y$, we evaluate the equation as $y\rightarrow\infty$, we have $p_x=0$.
We still using the stream-function of $[u,v]$
\begin{align}
\phi(x,y):=\int_0^y u(x,y')dy',\hspace{5 mm} \p_y\phi=u, \hspace{5 mm}\p_x\phi=-v.
\end{align}
Then $\phi$ satisfies
\begin{align}\label{phi}
\left
\{
\begin{aligned}
& \ub \phi_{xy} +\ub_x\phi_{y}+ \vb \phi_{yy}- \phi_x \ub_{y} - \phi_{yyy}  =  h, \\
& \phi|_{x = 0}=\int_0^yu_0(y')dy', \hspace{3 mm} \phi|_{y = 0} =\phi_y|_{y = 0}=0, \hspace{3 mm} \phi_y|_{y \rightarrow \infty} = 0.
\end{aligned}
\right.
\end{align}
In order to give a priori estimate of (\ref{phi}), we denote $g=\frac{\phi}{\ub}$. Recall $\ub\sim y$ when $y\leqslant1$ and $\ub\sim 1$, when $y\geqslant1$, and $\phi|_{y = 0} =\phi_y|_{y = 0}=0$, $g$ is well-defined. And $g$ satisfies
\begin{align}\label{g}
\left
\{
\begin{aligned}
&\p_{x}[\ub^2 g_y] - \p^3_{y}[\ub g]+\vb\p^2_{y}[\ub g]-\ub\vb_{yy}g  =  h, \\
& g|_{x = 0}=\frac{\int_0^yu_0}{\ub}, \hspace{3 mm} g|_{y = 0}=0, \hspace{3 mm} g_y|_{y \rightarrow \infty} = 0.
\end{aligned}
\right.
\end{align}
In the appendix, we write $\|\cdot\|$ for $L^2_{x,y}(\Omega)$, $\z\cdot,\cdot\y=\z\cdot,\cdot\y_{L^2_{x,y}}$, $\z\cdot,\cdot\y_{x=x_0}=\z\cdot,\cdot\y_{L^2_{y}(x=x_0)}$ and $\z\cdot,\cdot\y_{y=0}=\z\cdot,\cdot\y_{L^2_{x}(y=0)}$, now we define the norms of $g$:
\begin{align}\label{norm g}
\begin{aligned}
\|g\|_{\Xi_0}:=\sup_{0\leqslant x_0\leqslant L}\|\ub g_y\rho\|_{x=x_0}+\|\sqrt{\ub} g_{yy}\rho\|,\\
\|g\|_{\Xi_1}:=\sup_{0\leqslant x_0\leqslant L}\|\ub g_{xy}\frac{\rho}{\z y\y}\|_{x=x_0}+\|\sqrt{\ub} g_{xyy}\frac{\rho}{\z y\y}\|,
\end{aligned}
\end{align}
here $\rho=\z y\y^N$, for $N$ large constant. Next, let us prove the following priori estimate of $g$.
\begin{proposition}
Suppose $g$ be a smooth solution of (\ref{g}), $L>0$ small enough, then there exists a positive constant $C$ independent on $L$, s.t. $g$ satisfies
\begin{align}\label{g est1}
\|g\|_{\Xi_0}^2\leqslant &\|\ub g_y \rho\|_{x=0}^2+C\|h\rho\|^2,\\\label{g est2}
\|g_x\|_{\Xi_1}^2\leqslant &\|\ub g_{xy} \frac{\rho}{\z y\y}\|_{x=0}^2+C\|g\|_{\Xi_0}^2+C\|h_x\frac{\rho}{\z y\y}\|^2.
\end{align}
\end{proposition}
\begin{proof}
Multiply equation (\ref{g}) by $g_y\rho^2$ and integrate in $(0,x_0)\times(0,\infty)$.
\[
\begin{split}
\z[\ub^2 g_y]_x,g_y\rho^2\y=&\z \ub^2 g_{xy},g_y\rho^2\y+2\z\ub\ub_x g_y,g_y\rho^2\y\\
                           =&\frac{1}{2}\|\ub g_y\rho\|^2_{x=x_0}-\frac{1}{2}\|\ub g_y\rho\|^2_{x=0}+\z\ub\ub_x g_y,g_y\rho^2\y.
\end{split}
\]
We can dominate $\|g_y\|$ by $\|g\|_{\Xi_0}$. Let $0<\xi\leqslant1$ be a constant being choosing later. $\chi(y)$ is smooth cut-off function, satisfies
$\chi|_{[0,1]}=1$, $\chi|_{[2,\infty]}=0$. Then,
\[
\begin{split}
\z g_y,g_y\rho^2\y\lesssim\z g_y,g_y[1-\chi(\frac{y}{\xi})]^2\rho^2\y+\z g_y,g_y\chi(\frac{y}{\xi})^2\rho^2\y.
\end{split}
\]
When $y\leqslant1$, $1-\chi(\frac{y}{\xi})\lesssim\frac{y}{\xi}\lesssim\frac{\ub}{\xi}$, when $y>1$,  $1-\chi(\frac{y}{\xi})\lesssim\ub\lesssim\frac{\ub}{\xi}$. So
\[
\begin{split}
\z g_y,g_y[1-\chi(\frac{y}{\xi})]^2\rho^2\y\lesssim\frac{1}{\xi^2}\|\ub g_y\rho\|^2\lesssim \frac{L}{\xi^2}\|g\|_{\Xi_0}^2.
\end{split}
\]
While
\[
\begin{split}
\z g_y,g_y\chi(\frac{y}{\xi})^2\rho^2\y=&-2\z yg_y,g_{yy}\chi^2(\frac{y}{\xi})\rho^2\y-\frac{2}{\xi}\z yg_y,g_y\chi(\frac{y}{\xi})\chi'(\frac{y}{\xi})\rho^2\y-2\z yg_y,g_y\chi^2(\frac{y}{\xi})\rho\rho_y)\y\\
                                       \lesssim &\|y\chi(\frac{y}{\xi})g_{yy}\rho\|^2+\frac{1}{\xi^2}\|\ub g_y\rho\|^2\\
                                       \lesssim &\xi\|\sqrt{\ub}g_{yy}\rho\|^2+\frac{L}{\xi^2}\|g\|_{\Xi_0}^2.
\end{split}
\]
So we have
\[
\begin{split}
\z g_y,g_y\rho^2\y\lesssim \xi\|\sqrt{\ub}g_{yy}\rho\|^2+\frac{L}{\xi^2}\|g\|_{\Xi_0}^2,
\end{split}
\]
select $\xi=L^\frac{1}{3}$, then
\begin{align}\label{g hardy}
\z g_y,g_y\rho^2\y\lesssim L^\frac{1}{3}\|g\|_{\Xi_0}^2.
\end{align}
So the first term is
\begin{align}\label{g 1}
\z[\ub^2 g_y]_x,g_y\rho^2\y=\frac{1}{2}\|\ub g_y\rho\|^2_{x=x_0}-\frac{1}{2}\|\ub g_y\rho\|^2_{x=0}+O\big(L^\frac{1}{3}\|g\|_{\Xi_0}^2\big).
\end{align}
The second term:
\[
\begin{split}
-\z\p^3_y[\ub g],g_y\rho^2\y=\z\p^2_y[\ub g],g_{yy}\rho^2\y+2\z\p^2_y[\ub g],g_{yy}\rho_y\rho\y+\z\p^2_y[\ub g],g_{y}\rho^2\y_{y=0}.
\end{split}
\]
\[
\begin{split}
\z\p^2_y[\ub g],g_{yy}\rho^2\y=&\z\ub g_yy+2\ub_y g_y+\ub_{yy}g, g_{yy}\rho^2\y\\
                              =&\|\sqrt{\ub}g_{yy}\rho\|^2-\z\ub_yg_y,g_y\y_{y=0}+\z(\ub_y\rho^2)_y, g^2_y\y\\
                               &-\z\ub_{yy}g_y, g_{y}\rho^2\y-\z(\ub_{yy}\rho^2)_yg,g_y\y\\
                              =&\|\sqrt{\ub}g_{yy}\rho\|^2-\z\ub_yg_y,g_y\y_{y=0}+O\big(\|g_y\rho^2\|+\|y(\ub_{yy}\rho^2)_y\|_{L^\infty}\|\frac{g}{y}\| \|g_y\|\big)\\
                              =&\|\sqrt{\ub}g_{yy}\rho\|^2-\z\ub_yg_y,g_y\y_{y=0}+O\big(\|g_y\rho\|^2\big),
\end{split}
\]
\[
\begin{split}
2\z\p^2_y[\ub g],g_{yy}\rho_y\rho\y=&\z\ub (g_y^2)_y, \rho_y\rho\y+4\z\ub_yg_y,g_y\rho_y\rho\y+2\z\ub_{yy}g,g_y\rho_y\rho\y\\
                                   =&\z(\ub\rho_y\rho,g^2_y\y+O\big(\|g_y\rho^2\|+\|y\ub_{yy}\rho_y\rho\|_{L^\infty}\|\frac{g}{y}\| \|g_y\|\big))\\
                                   =&O\big(\|g_y\rho\|^2\big),
\end{split}
\]
\[
\begin{split}
\z\p^2_y[\ub g],g_{y}\rho^2\y_{y=0}=2\z\ub_yg_y,g_y\rho^2\y_{y=0}.
\end{split}
\]
So the second term is
\begin{align}\label{g 2}
\z\p^2_y[\ub g],g_{yy}\rho^2\y=\|\sqrt{\ub}g_{yy}\rho\|^2+\z\ub_yg_y,g_y\rho^2\y_{y=0}+O\big(L^\frac{1}{3}\|g\|_{\Xi_0}^2\big).
\end{align}
The third term is
\begin{align}\label{g 3}
\begin{aligned}
\z\vb(\ub g)_{yy},g_y\rho^2\y=&\z\vb(\ub g_yy+2\ub_y g_y+\ub_{yy}g),g_y\rho^2\y\\
                             =&-\frac{1}{2}\z(\vb\ub\rho^2)_y,g^2_y\y+O\big(\|\vb \ub_y \|_{L^\infty}\|g_y\rho\|^2+\|y\vb\ub_{yy}\rho^2\|_{L^\infty}\|\frac{g}{y}\|\|g_y\|\big)\\
                             =&O\big(\|\frac{(\vb\ub\rho^2)_y}{\rho^2} \|_{L^\infty}\|g_y\rho\|^2+\|g_y\rho\|^2\big)\\
                             =&O\big(L^\frac{1}{3}\|g\|_{\Xi_0}^2\big).
\end{aligned}
\end{align}
And the last one is
\begin{align}\label{g 4}
\begin{aligned}
\z\vb_{yy}\ub g,g_y\rho^2\y=O\big(\|yvb_{yy}\ub\rho^2\|_{L^\infty}\|\frac{g}{y}\|\| g_y\|\big)=O\big(\|g_y\rho\|^2\big).
\end{aligned}
\end{align}
Collect (\ref{g 1}), (\ref{g 2}), (\ref{g 3}), (\ref{g 4}), we have
\begin{align}\label{g what?}
\begin{aligned}
\frac{1}{2}\|\ub g_y\rho\|^2_{x=x_0}+\|\sqrt{\ub}g_{yy}\rho\|^2+\z\ub_yg_y,g_y\rho^2\y_{y=0}=O\big(L^\frac{1}{3}\|g\|_{\Xi_0}^2\big)+\frac{1}{2}\|\ub g_y\rho\|^2_{x=0}+\z h,g_y\rho^2\y.
\end{aligned}
\end{align}
Take the supremum of $0\leqslant x_0\leqslant L$, notice that $L$ small enough,
\begin{align}\label{g es}
\begin{aligned}
\sup_{0\leqslant x_0\leqslant L}\|\ub g_y\rho\|^2_{x=x_0}+\|\sqrt{\ub}g_{yy}\rho\|^2+\z\ub_yg_y,g_y\rho^2\y_{y=0}\lesssim\|\ub g_y\rho\|^2_{x=0}+\| h\rho\|^2.
\end{aligned}
\end{align}

The inequality (\ref{g est2}) is similar to the (\ref{g est1}). Differential equation (\ref{g}) respect to $x$,
\begin{align}\label{g x}
\begin{aligned}
&\p_{x}[\ub^2 g_{xy}]- \p^3_{y}[\ub g_x]+\vb\p^2_{y}(\ub g_x)-\ub\vb_{yy}g_x+\p_{x}[2\ub\ub_x g_{y}]- \p^3_{y}[\ub_x g]\\
&+\vb_x\p^2_{y}(\ub g)+\vb\p^2_{y}(\ub_x g)-\ub_x\vb_{yy}g-\ub\vb_{xyy}g =  h_x.
\end{aligned}
\end{align}
Take $g_{xy}\frac{\rho^2}{\z y\y^2}$ as the test function, like (\ref{g what?}),
\begin{align}\label{g x1}
\begin{aligned}
 &\z\p_{x}[\ub^2 g_{xy}]- \p^3_{y}[\ub g_x]+\vb\p^2_{y}(\ub g_x)-\ub\vb_{yy}g_x, g_{xy}\frac{\rho^2}{\z y\y^2}\y\\
=&\frac{1}{2}(\|\ub g_{xy}\frac{\rho}{\z y\y}\|^2_{x=x_0}-\|\ub g_{xy}\frac{\rho}{\z y\y}\|^2_{x=0})+\|\sqrt{\ub}g_{xyy}\frac{\rho}{\z y\y}\|^2+\z\ub_yg_{xy},g_{xy}\frac{\rho}{\z y\y}\y_{y=0}+O\big(L^\frac{1}{3}\|g\|_{\Xi_1}^2\big).
\end{aligned}
\end{align}
While
\begin{align}\label{g x2}
\begin{aligned}
  &\z\p_{x}[2\ub\ub_x g_{y}]+\vb_x\p^2_{y}(\ub g)+\vb\p^2_{y}(\ub_x g)-\ub_x\vb_{yy}g-\ub\vb_{xyy}g,g_{xy}\frac{\rho^2}{\z y\y^2}\y\\
 =&O\big(\|\sqrt{\ub} g_{yy}\rho\|^2+\|g_y\rho\|^2+\|g_{xy}\frac{\rho}{\z y\y}\|^2\big)\\
 =&O\big(\|g\|_{\Xi_0}^2+L^\frac{1}{3}\|g\|_{\Xi_1}^2\big).
\end{aligned}
\end{align}
The difficult term is
\begin{align}\label{g dif}
\begin{aligned}
\z- \p^3_{y}[\ub_x g],g_{xy}\frac{\rho^2}{\z y\y^2}\y=&-\z\ub_xg_{yyy}+3\ub_{xy}g_{yy},g_{xy}\frac{\rho^2}{\z y\y^2}\y+O\big(\|g_y\rho\|^2+\|g_{xy}\frac{\rho}{\z y\y}\|^2\big)\\
                                                     =&O\big((\|\ub g_{yyy}\frac{\rho}{\z y\y}\|+\|g_{yy}\frac{\rho}{\z y\y}\|)\|g_{xy}\frac{\rho}{\z y\y}\|+\|g_y\rho\|^2+\|g_{xy}\frac{\rho}{\z y\y}\|^2 \big).
\end{aligned}
\end{align}
From equation (\ref{phi}), we have
\[
\begin{split}
\|\phi_{yyy}\frac{\rho}{\z y\y}\|^2=O\big(\|g\|_{\Xi_0}^2+L^\frac{1}{3}\|g\|_{\Xi_1}^2+\|h\frac{\rho}{\z y\y}\|^2\big),
\end{split}
\]
notice that the fact
\[
\begin{split}
\|\p^2_y(\frac{\phi}{y})\|_{L^2_{loc}}=O\big(\|\phi_{yyy}\|_{L^2_{loc}}+\|\phi_{yy}\|_{L^2_{loc}}\big),
\end{split}
\]
similarly, we can get
\[
\begin{split}
\|g_{yy}\|_{L^2_{loc}}=O\big(\|\phi_{yyy}\|_{L^2_{loc}}+\|\phi_{yy}\|_{L^2_{loc}}\big),
\end{split}
\]
so we have
\[
\begin{split}
\|g_{yy}\frac{\rho}{\z y\y}\|^2=&\|g_{yy}\chi\|^2+\|g_{yy}(1-\chi)\frac{\rho}{\z y\y}\|^2\\
            =&O\big(\|\phi_{yyy}\|^2_{L^2_{loc}}+\|\phi_{yy}\|^2_{L^2_{loc}}+\|\sqrt{\ub}g_{yy}\frac{\rho}{\z y\y}\|^2\big),\\
            =&O\big(\|g\|_{\Xi_0}^2+L^\frac{1}{3}\|g\|_{\Xi_1}^2+\|h\frac{\rho}{\z y\y}\|^2\big),
\end{split}
\]
and
\[
\begin{split}
\|\ub g_{yyy}\frac{\rho}{\z y\y}\|^2=&\|(\phi_{yyy}-3\ub_y g_{yy}-3\ub_{yy}g_{y}-\ub_{yyy}g)\frac{\rho}{\z y\y}\|^2\\
                                    =&O\big(\|g\|_{\Xi_0}^2+L^\frac{1}{3}\|g\|_{\Xi_1}^2+\|h\frac{\rho}{\z y\y}\|^2\big).
\end{split}
\]
We conclude (\ref{g dif}) as
\begin{align}\label{g x3}
\begin{aligned}
\z- \p^3_{y}[\ub_x g],g_{xy}\frac{\rho^2}{\z y\y^2}\y=O\big(\|g\|_{\Xi_0}^2+L^\frac{1}{3}\|g\|_{\Xi_1}^2+\|h\frac{\rho}{\z y\y}\|^2\big).
\end{aligned}
\end{align}
Collect (\ref{g x1}), (\ref{g x2}), (\ref{g x3}), we have
\begin{align}\label{gx what?}
\begin{aligned}
&\frac{1}{2}\|\ub g_{xy}\frac{\rho}{\z y\y}\|^2_{x=x_0}+\|\sqrt{\ub}g_{xyy}\frac{\rho}{\z y\y}\|^2+\z\ub_yg_{xy},g_{xy}\frac{\rho}{\z y\y}\y_{y=0}\\
=&\frac{1}{2}\|\ub g_{xy}\frac{\rho}{\z y\y}\|^2_{x=0}+O\big(\|g\|_{\Xi_0}^2+L^\frac{1}{3}\|g\|_{\Xi_1}^2+\|h_x\frac{\rho}{\z y\y}\|^2+\|h\frac{\rho}{\z y\y}\|^2\big).
\end{aligned}
\end{align}
So we finish the proof of (\ref{g est2}).
\end{proof}

\noindent \textbf{Acknowledgements:} L. Zhang is partially supported by NSFC under grant
11471320 and 11631008.

\bibliographystyle{springer}
\bibliography{mrabbrev,literatur}
\newcommand{\noopsort}[1]{} \newcommand{\printfirst}[2]{#1}
\newcommand{\singleletter}[1]{#1} \newcommand{\switchargs}[2]{#2#1}

\end{document}